\documentclass[11pt]{article}
\usepackage{amsmath, amsthm, amssymb, amsfonts, mathrsfs, mathtools}
\usepackage{graphicx}
\usepackage{makeidx}
\usepackage[left=2cm,right=2cm,top=2cm,bottom=2cm]{geometry}
\usepackage{caption}
\usepackage{subcaption}
\usepackage[section]{placeins}
\usepackage{enumitem}
\usepackage{tikz}
\usepackage{stmaryrd}
\usepackage{authblk}
\allowdisplaybreaks

\usepackage{tikz}
\usetikzlibrary{arrows.meta}

\usetikzlibrary{decorations.pathreplacing}
\tikzstyle{vertex}=[auto=left,circle,draw=black,fill=white, inner sep=1.5]

\usepackage{hyperref}
\hypersetup{colorlinks=true, citecolor={blue}, linkcolor=blue, filecolor=magenta, urlcolor=cyan}

\newtheorem{theorem}{Theorem}[section]
\newtheorem{prop}[theorem]{Proposition}
\newtheorem{lema}[theorem]{Lemma}
\newtheorem{corollary}[theorem]{Corollary}
\newtheorem{definition}{Definition}
\newtheorem{ex}{Example}[section]  

\newtheorem{remark}[theorem]{Remark}

\def \x {{\mathbf x}}

\title{Nonexistence results of generalized bent functions from $\mathbb{Z}_3^n$ to $ \mathbb{Z}_m$}

\author[1]{Priya Dhankhar}
\author[2]{Sanjay Kumar Singh}
\affil[ ]{\small{\textsuperscript{1,2}Department of Mathematics, Indian Institute of Science Education and Research Bhopal, India.}}
\affil[ ]{ {\textsuperscript{1}priya22@iiserb.ac.in}
\textsuperscript{2}sanjayks@iiserb.ac.in}

\date{}

\begin{document}
	\maketitle
	
	\vspace{-0.3in}
	
\begin{center}{\textbf{Abstract}}\end{center} 

In this paper, we investigate generalized bent functions (GBFs) from $\mathbb{Z}_3^n$ to $\mathbb{Z}_m$. We show that GBFs exist whenever $3$ divides $m$, while several nonexistence results are obtained when $3\nmid m$. In particular, we prove that no GBFs exist for $n=1,2$ when $m$ is odd and not divisible by $3$. For the case $n=3$, we establish the nonexistence of GBFs $f:\mathbb{Z}_3^3 \rightarrow \mathbb{Z}_{5\cdot11^r}$ for all nonnegative integers $r$. Finally, we show that no GBF exists from $\mathbb{Z}_3$ to $\mathbb{Z}_{2m'}$ and $\mathbb{Z}_3^2$ to $\mathbb{Z}_{2m'}$, where $m'$ is odd and not divisible by $3$.

\vspace*{0.3cm}

\section{Introduction}
The study of bent functions originated in the work of Rothaus \cite{rothaus1976bent}, where Boolean functions $f:\mathbb{Z}_2^n \rightarrow \mathbb{Z}_2$ were considered. It is well known that a Boolean bent function 
$f:\mathbb{Z}_2^n \rightarrow \mathbb{Z}_2$ exists if and only if $n$ is even (see \cite{rothaus1976bent}). Since then, the concept of bent functions has been generalized and investigated in several settings 
(see \cite{tokareva2011generalizations} for a detailed survey). In particular, bent functions can be studied  over arbitrary finite abelian groups.

 Let $G$ be a finite abelian group. Let $m \in \mathbb{Z}_{>0}$ and $\zeta_m$ denote a primitive $m$-th root of unity. Let $\widehat{G}$ denote the character group of $G$.
\begin{definition}
  A function $f:G \rightarrow \mathbb{Z}_m$ is called a \textbf{Generalized Bent function (GBF)} if
    \begin{align}\label{fourier_transformation}
        \left | \sum_{x\in G} \zeta_{m}^{f(x)} \chi(x) \right |^{2} =|G|\hspace{2mm} \text{for every $\chi \in \widehat{G}$}.
    \end{align}

\end{definition}

\noindent For $G=\mathbb{Z}_2^n$ and $m=2$, we obtain the Boolean bent function studied by Rothaus. For $G=\mathbb{Z}_q^n$ and $m=q$,  Kumar, Scholtz, and Welch investigated functions of the form $f : \mathbb{Z}_q^n \to \mathbb{Z}_q$ and established the following fundamental existence results in the theory of GBFs. 

\begin{theorem} \label{P.V.Kumar_existence_resut}
    \cite[Section $IV$]{kumar1985generalized} There exists a GBF 
    $\mathbb{Z}_q^n \rightarrow \mathbb{Z}_q$  whenever $n$ is even or $q \not\equiv 2 \mod 4$.
\end{theorem} 

Later, Schmidt \cite{schmidt2009quaternary} investigated GBFs 
$f:\mathbb{Z}_2^n \rightarrow \mathbb{Z}_m$ motivated by their applications in CDMA communications.  Various constructions of such GBFs can be found in  \cite{schmidt2009quaternary, tang2017complete, schmidt2009bbz}. The wide range of applications of GBFs in coding theory, cryptography, and information theory makes them an important topic of study.
Moreover, their connections with other combinatorial objects such as relative difference sets and group-invariant Butson Hadamard matrices highlight their significance in combinatorial theory, see \cite{schmidt2019survey} for further details.  

 Note that the condition in Theorem \ref{P.V.Kumar_existence_resut} is sufficient but need not be necessary.  Indeed, when $n$ is odd and $q \equiv 2 \pmod{4}$, GBFs do not necessarily exist. In this setting, several nonexistence results have been established under additional assumptions; see \cite{li2017nonexistence, leung2019nonexistence, lv2025nonexistence}. In recent years, considerable attention has also been devoted to the nonexistence problem for GBFs from $\mathbb{Z}_2^n$ to $\mathbb{Z}_m$; see 
\cite{liu2017nonexistence, leung2020new, leung2023nonexistence}. To the best of our knowledge, no work has been carried out on GBFs from $\mathbb{Z}_3^n$ to $\mathbb{Z}_m$.

\medskip
This paper investigates the existence and nonexistence of GBFs from $\mathbb{Z}_3^n$ to $\mathbb{Z}_m$. Note that the case where $3$ divides $m$ can be resolved easily using known existence results. Consequently, the more interesting and nontrivial situation arises when $3$ does not divide $m$. In this paper, we establish several nonexistence results for GBFs $f : \mathbb{Z}_3^n \to \mathbb{Z}_m$. Our main results are as follows.
\medskip

\noindent \textbf{Theorem} \textit{(Theorem \ref{Nonexistence_of _gbf_for_odd_m})}
Let there exists a GBF $f : \mathbb{Z}_3^n \to \mathbb{Z}_m$, where $m = \prod_{j=1}^t p_j^{\alpha_j}$ with primes satisfying $5 \leq p_1 < p_2 < \cdots < p_t$ and $\alpha_j$ are positive integers, then $t \geq 2$ and $2p_1 + p_2 \leq 3^n$.

\medskip
\noindent \textbf{Corollary} \textit {(Corollary \ref{non_exist_for_n_1_2_and_m_odd})}
 Let $m$ be an odd integer not divisible by $3$. If $n = 1$ or $n = 2$, then there does not exist a GBF $f : \mathbb{Z}_3^n \to \mathbb{Z}_m$.

\medskip
\noindent \textbf{Corollary} \textit {(Corollary \ref{non exist_for_n=3_m_odd})}
There does not exist a GBF $f :\mathbb{Z}_3^3 \rightarrow \mathbb{Z}_{5\cdot 11^r}$. 

\medskip
 \noindent  \textbf{Theorem} \textit {(Theorem \ref{NO_GBF_for_2^k})}
There does not exist a GBF  $ f: \mathbb{Z}_3^n \rightarrow \mathbb{Z}_{2^k}$ for any positive integer $n$.

\medskip
\noindent \textbf{Corollary} \textit {(Corollary \ref{non_exist_n=1_mis even})}
There does not exist a GBF $f: \mathbb{Z}_3 \to \mathbb{Z}_{2m'}$, where $m'$ is an odd positive integer not divisible by $3$.

\medskip
\noindent \textbf{Theorem} \textit {(Theorem \ref{non_exist_n=2and_m_is_even})} 
There does not exist a GBF $f: \mathbb{Z}_3^2 \to \mathbb{Z}_{2m'}$, where $m'$ is an odd positive integer not divisible by $3$.

\medskip
The paper is organised as follows. In Section $2$, we collect preliminary results on group rings and characters, along with results on vanishing sums of roots of unity, which are crucial for our analysis. In Section $3$, we prove  Theorem~\ref{Nonexistence_of _gbf_for_odd_m} and, as a consequence, we obtain Corollary~\ref{non_exist_for_n_1_2_and_m_odd}. Also, for $n=3$, we determine the possible forms of autocorrelation functions and then use them to prove a nonexistence result for GBFs from $\mathbb{Z}_3^3$ to $\mathbb{Z}_{5 \cdot 11^k}$, where $k$ is a positive integer.

In Section~4, we consider the case where $m$ is even and not divisible by $3$. We first prove Theorem~\ref{NO_GBF_for_2^k}. Later in this section, we establish the nonexistence results stated in Corollary~\ref{non_exist_n=1_mis even} and Theorem~\ref{non_exist_n=2and_m_is_even}.

%%%%%%%%%%%%%%%%%%%%%%%%%%%%%%%%%%%%%%%%%%%%%%%%%%%%%%%%%%%%%%%%%%%%%%%%%%%%%%%%%%%%%%%%%%%%%%%%%%%%%%%%%%%%%%%%%%%%%%%%%%%%%%%%%%%%%%%%%%%%%%%%%%%%%%%%%%%%%%%%%%%%%%%%%%%%%%%%%%%%%%%%%%%%%%%%%%%%%%%%%%%%%%%%%%%%%%%%%%%%%%%%%%%%%%%%%%%%%%%%%%%%%%%%%%%%%%%%%%%%%%%%%%%%%%%%%%%%%%%%%%%%%%%%%%%%%%%%%%%%%%%%%%%%%%%%%%%%%%%%%
\section{Preliminaries}
%%%%%%%%%%%%%%%%%%%%%%%%%%%%%%%%%%%%%%%%%%%%%%%%%%%%%%%%%%%%%%%%%%%%%%%%%%%%%%%%%%
This section collects the algebraic and combinatorial tools used in the later nonexistence arguments. We first recall the required facts on group rings and characters, then reformulate GBFs in terms of autocorrelation functions. Finally, we recall the results on vanishing sum together with the existence results that isolate the case $3\nmid m$.
\subsection{Group ring and character theory}

Group rings and characters of abelian groups play a significant role in the study of generalized bent functions.  
Let $G$ be a finite abelian group of order $v$, and $1_G$ denote the identity element of $G$. A character of $G$ is a group homomorphism $\chi : G \rightarrow \mathbb{C}^*$, that is,
\[
\chi(ab)= \chi(a)\chi(b) \quad \text{for all } a,b \in G.
\] 
Let $\widehat{G}$ denote the set of all characters of $G$. Then $\widehat{G}$ forms a group under pointwise multiplication, defined by
\[
(\chi \chi')(g) = \chi(g)\chi'(g), \quad \text{for all } g \in G,
\]
and is isomorphic to $G$.  The character $\chi_0$ with $\chi_0(g)=1$ for all $g \in G$ is known as the principal character, and it is the identity element of $\widehat{G}$. Since
\[
\chi(a)^v = \chi (a^v) =\chi(1_G)=1 \hspace{3mm}  \text{for all} \hspace{3mm} a \in G.
\]
Therefore, $\chi(a)$ is a root of unity. 
The group ring is a useful tool in the theory of characters. Let $R$ be a ring. The group ring $R[G]$ is an $R$-module with the elements of $G$ forming a basis.  
Each element of $R[G]$ can be written as  
\[
D = \sum_{g \in G} a_g g, \quad a_g \in R,
\]  
where the scalars $a_g$ are called the \textit{coefficients} of $D$. The \textit{support} of a group ring element $D$ is defined by  
\[
\operatorname{supp}(D) = \{ g \in G : a_g \neq 0 \}.
\]

\noindent For $R = \mathbb{C}$ and $D = \sum_{g \in G} a_g g \in \mathbb{C}[G]$, we define  
\[
D^{(-1)} = \sum_{g \in G} \overline{a_g}\, g^{-1},
\]  
where $\overline{a_g}$ denotes the complex conjugate of $a_g$.  
We also define  
\[
|D| = \sum_{g \in G} a_g 
\quad \text{and} \quad 
\|D\| = \sum_{g \in G} |a_g|.
\]
When $R = \mathbb{Z}$, a partial order can be defined on the elements of the group ring $\mathbb{Z}[G]$. Let $D=\sum_{g \in G} a_g g$ and $D'=\sum_{g \in G} b_g g$ be two elements of $\mathbb{Z}[G]$, we write $D\geq D'$ whenever $a_g \geq b_g$ for every $g\in G$.\\
Define $\mathbb{N}[G]$ by
\begin{align*}
    \left\{ \sum_{g\in G}a_ig~|~ a_i \in \mathbb{Z},~ a_i \geq 0 \right\}.
\end{align*}

Furthermore, any group homomorphism $\phi : G \to G$ extends naturally to a ring homomorphism on $R[G]$, given by  
\[
\phi(D) = \sum_{g \in G} a_g \phi(g).
\]

\noindent For $D = \sum_{g \in G} a_g g$ and $\chi \in \widehat{G}$, define
\[
\chi(D) = \sum_{g \in G} a_g \chi(g).
\]
The following result is known as the Fourier inversion formula.
\begin{theorem}\cite{beth1999design}\label{Fourier Inversion Formula}
    Let $G$ be a finite abelian group and $D=\sum_{g\in G}a_gg$ be an element of the group ring $\mathbb{C}[G]$. Then the coefficients of $D$ are given as 
    \begin{align*}
        a_g = \frac{1}{|G|}\sum_{\chi \in \widehat{G}}\chi (Dg^{-1}).
    \end{align*}
Moreover, if $D$ and $D'$ are elements of the group ring $\mathbb{C}[G]$  with $\chi(D)=\chi(D')$ for all characters $\chi$ of $G$, then $D=D'$.
\end{theorem}
\begin{proof}
    A proof can be found in \cite[Ch.~VI, Lemma~3.5]{beth1999design}.
\end{proof}
%%%%%%%%%%%%%%%%%%%%%%%%%%%%%%%%%%%%%%%%%%%%%%%%%%%%%%%%%%%%%%%%%%%%%%%%%%
%%%%%%%%%%%%%%%%%%%%%%%%%%%%%%%%%%%%%%%%%%%%%%%%%%%%%%%%%%%%%%%%%%%%%%%%%%%%%%%%%%%%%
\subsection{GBFs in terms of autocorrelation function}
In this section, we use group rings and character theory to describe generalized bent functions (GBFs) in terms of group ring elements, following the approach of \cite{leung2020new}. Let $C_m$ denote the multiplicative cyclic group of order $m$, and let $g$ be a generator of $C_m$.

\begin{definition}
   Let $f: G \rightarrow \mathbb{Z}_m$ be a function. Define an element $B_f$ of the group ring $\mathbb{Z}[\zeta_m][G]$ by
   \begin{align}
       B_f:= \sum_{x\in G}\zeta_m^{f(x)}x.
       \end{align}
\end{definition}
\noindent For a finite abelian group $G$, we may identify $\widehat{G} = \{ \chi_v \mid v \in G \}$. Moreover,
\begin{align}\label{chi_v and fourier transformation}
\chi_v(B_f) = \sum_{x \in G} \zeta_m^{f(x)} \chi_v(x).
\end{align}
Thus, a function $f : G \to \mathbb{Z}_m$ is a GBF if and only if $|\chi_v(B_f)|^2 = |G|$ for all $v \in G$.

\begin{prop}\label{equiv_def_of_gbf}
    If $f: G \rightarrow \mathbb{Z}_m$ be a function. Then $f$ is a GBF if and only if 
    \begin{align}
        B_f B_f^{(-1)} =|G|\hspace{2mm}\text{in} \hspace{2mm} \mathbb{Z}[\zeta_m][G].
    \end{align}
\end{prop}
\begin{proof}
   Note that $B_f^{(-1)} = \sum \limits_{x \in G}\overline{\zeta_m^{f(x)}}x^{-1}$ and hence $\chi_v(B_f^{(-1)})=\overline{\chi_v(B_f)}$. Therefore,
   \begin{align*}
       |\chi_v(B_f)|^{2}= \chi_v(B_f)\overline{\chi_v(B_f)}= \chi_v(B_fB_f^{(-1)}).
   \end{align*}
   From Equation (\ref{chi_v and fourier transformation}), it follows that $f$ is GBF if and only if $\chi_v(B_fB_f^{-1}) = |G|$ holds for all $v \in G$. The result now follows from the Fourier inversion formula (Theorem~\ref{Fourier Inversion Formula}).
\end{proof}
\begin{definition}
     Let $f : G \rightarrow C_m$ be a function. Define an element $D_f$ of the group ring $\mathbb{Z}[C_m][G]$ by
     \begin{align*}
         D_f := \sum_{x\in G}g^{f(x)}x.
     \end{align*}
\end{definition}
\noindent Note that
\begin{align}\label{D_fD_f-1 =_sum_of E_x}
    D_fD_f^{(-1)}= \sum_{x\in G}\sum_{y \in G}g^{f(x+y)-f(y)}x =\sum_{x\in G}E_xx,
\end{align}
where
\[
E_x = \sum_{y \in G} g^{f(x+y)-f(y)} \in \mathbb{Z}[C_m].
\]
The element $E_x$ is referred to as the autocorrelation function of $f$ at $x \in G$.

\begin{lema} \label{E_x_is_a_vanishing_sum}
    Let $f: G \rightarrow \mathbb{Z}_m$ be a GBF. Then for each $\tau \in \widehat{C_m}$ with $\rm{ord}(\tau)=m$, we have $\tau(E_x)=0$ for all $x\in G\setminus {1_G}$.
\end{lema}    
\begin{proof}
    Observe that $\tau(g)$ is a primitive $m^{th} $ root of unity and $\tau(D_f)=B_f$. We have 
    \begin{align*}
        \tau(D_fD_f^{(-1)})= \tau(D_f)\tau(D_f^{(-1)}) = \sum_{x\in G}\tau(E_x)x.
    \end{align*}
    Thus,
    \begin{align*}
        \sum_{x\in G}\tau(E_x)x= B_fB_f^{(-1)}. 
    \end{align*}
Since $f$ is GBF, Proposition \ref{equiv_def_of_gbf} implies that, $\sum_{x\in G}\tau(E_x)x =|G|$. As $\tau(E_{1_G})=|G|$, comparing coefficients in the group ring, we obtain $\tau(E_{1_G}) = |G|$ and $\tau(E_x) = 0$ 
for all $x \neq 1_G$.
\end{proof}
%%%%%%%%%%%%%%%%%%%%%%%%%%%%%%%%%%%%%%%%%%%%%%%%%%%%%%%%%%%%%%%%%%%%%%%%%%
%%%%%%%%%%%%%%%%%%%%%%%%%%%%%%%%%%%%%%%%%%%%%%%%%%%%%%%%%%%%%%%%%%%%%%%%%%%%%%%%%%%%%%%%%%%%%%%%%%%%%%
\subsection{Vanishing sum in $\mathbb{N}[C_m]$}
We now recall the terminology for vanishing sums of roots of unity that will be applied to the autocorrelation elements $E_x$.
Let $S~=~\sum_{i\in I}a_i\eta_i \in \mathbb{N}[\mu_m]$, where $\mu_m$ denotes the group of all $m^{\text{th}}$ roots of unity, $\eta_i$ are roots of unity, and all $a_i$ are the nonzero positive integers. Then we define the following:
\begin{enumerate}[label=(\roman*)]
    \item The \textbf{exponent} of $S$ is the smallest positive integer $e$ such that $\eta_i^e = 1$ for all $i \in I$.  
    \item The \textbf{reduced exponent} of $S$ is the smallest positive integer $r$ such that there exists some $j\in I$ with $(\eta_i \eta_j^{-1})^r = 1$ for all $i \in I$.  
     \item The \textbf{length} $l$ of $S$ is the number of roots of unity appearing in the expression for $S$, counted with multiplicity.
     \item If $S = 0$, then $S$ is called a \textbf{vanishing sum of roots of unity}.  
    \item The relation 
    $
    S = \sum_{i \in I} a_i \eta_i = 0
    $
    is said to be \textbf{minimal} if no proper subsum of the $\eta_i$ vanishes, where $i \in I$.  
\end{enumerate}

\begin{definition}\cite{leung2023nonexistence}
     Suppose that $X =\sum_{i=1}^ma_ig^{i} \in \mathbb{N}[C_m]$ . We say that $X$ is $v$-sum in $\mathbb{N}[C_m]$ if $\sum_{i=1}^ma_i\zeta_m^{i} $ is a vanishing sum of roots of unity. Furthermore, $X$ is called a minimal $v$-sum if $\sum_{i=1}^{m}b_ig^i \neq 0$ whenever $0 \leq b_i \leq a_i$ for all $i$ and $b_j < a_j$ for some $j$.
\end{definition}
\noindent Note that if $X\in \mathbb{N}[C_m]$ is a $v$-sum and $\tau$ is a character of $C_m$ of order $m$ then $\tau(X)=0$. We now define the reduced exponent for the $X$.
\begin{definition}
    Let $X=\sum_{i=1}^{m}a_ig^{i} \in \mathbb{N}[C_m]$ be a minimal $v$-sum. Then the reduced exponent of $X$ is defined as the reduced exponent of $\sum_{i=1}^ma_i\zeta_m^i$.
\end{definition}
\noindent The following lemma appears as Lemma 3 in \cite{leung2020new}. For the sake of completeness, we add the proof here.
\begin{lema}
Suppose $X \in \mathbb{N}[C_m]$ is a $v$-sum. Then $X$ can be written as $X = \sum X_i$,
where each $X_i$ is a minimal $v$-sum in $\mathbb{N}[C_m]$.
\end{lema}
\begin{proof}
If $X$ is a minimal $v$-sum, then there is nothing to prove. 
Otherwise, there exists a nonzero proper minimal $v$-subsum $X_1$ of $X$. Then $X-X_1$ is again a $v$-sum in $\mathbb{N}[C_m]$. 
If $X-X_1$ is minimal, then we are done. Otherwise, we may choose a nonzero proper minimal $v$-subsum $X_2$ of $X-X_1$. Continuing in this way, we obtain
\[
X=X_1+X_2+\cdots+X_r,
\]
where each $X_i$ is a minimal $v$-sum. The process terminates after finitely many steps because the length strictly decreases at each step.
\end{proof}

\noindent Note that such a decomposition does not need to be unique.
\begin{ex} \label{ex_of_dist_decomp_of_vsum} Let $g$ be a generator of $C_6$. It is easy to check that $X= \sum_{i=1}^6g^i$ is a  $v$-sum in $\mathbb{N}[C_6]$. Now, we have the following decompositions of $X$ as minimal $v$-sums.
\begin{align*}
    X & = \sum_{i=1}^6 g^i=(1+g^3) + g(1+g^3)+g^2(1+g^3),\\
    X &= \sum_{i=1}^6 g^i = (1+g^2+g^4) + g(1+g^2+g^4).
\end{align*}

\end{ex}
\noindent We have an extension of the reduced exponent to Vanishing $v$-sum in $\mathbb{N}[C_m]$.
Let $X$ be a vanishing $v$-sum and $\sum_{i\in I}X_i$ be a decomposition of $X$ into minimal $v$-sum with $k_i$ being the reduced exponent of $X_i$. Then $ \rm{lcm}\{k_i ~|~i\in I\}$ is the reduced exponent of $X$ with the decomposition $\sum_{i\in I} X_i$. In Example \ref{ex_of_dist_decomp_of_vsum}, the reduced exponent of $X$ with the decomposition $(1+g^3) + g(1+g^3)+g^2(1+g^3)$ is $2$ and it is $3$ with the decomposition $(1+g^2+g^4) + g(1+g^2+g^4)$.
\begin{definition}
    The $c$-exponent of a $v$-sum $X$ is the smallest positive integer among all reduced exponents of $X$ obtained from the distinct decomposition of $X$ into minimal $v$-sums.
\end{definition}
\noindent If $k$ is the reduced exponent of a minimal $v$-sum $X$ in $\mathbb{N}[C_m]$ then $k$ divides $m$. Moreover, the $c$-exponent of a $v$-sum $X$ in $\mathbb{N}[C_m]$ also divides $m$.
\begin{lema}\label{C_expo_is_square_free}
    Let $X \in \mathbb{N}[C_m]$ be a $v$-sum with $c$- exponent $k$. Write $m=\prod_{i=1}^tp_i^{\alpha_i}$, where $p_i$ are distinct primes. Then $k$ is square free. Moreover, $k=\prod_{i\in I} p_i$ for some $I\subseteq \{1,\ldots,t\}$.
\end{lema}
\begin{proof}
    The reduced exponent of a minimal $v$-sum is square free, as stated in \cite[Theorem 5]{conway1976trigonometric}. Therefore, the proof follows directly from the definition of $c$-exponent.
\end{proof}
\noindent For a divisor $t$ of $m$, let $D_t$ denote the unique cyclic subgroup of $C_m$ of order $t$. We identify $D_t$ with the group ring element $\sum_{d\in D_t}d$.
\begin{prop} \label{Discription_of_v-sum}
    Let $X \in \mathbb{N}[C_m]$ be a $v$-sum with the $c$-exponent $k$. Then the following holds;
    \begin{enumerate}[label=(\roman*)]
        \item \cite[Lemma 4]{leung2020new} $\lVert X \rVert \geq 2+ \sum_{p|k}(p-2)$.
        \item  \cite[ Proposition $18(a)$]{leung2023nonexistence} If $k=p$, then $X=YD_p$, where $Y\in \mathbb{N}[C_m]$. 
        \item \cite [Proposition $18(b)$]{leung2023nonexistence} If $k=p_1p_2$, then $X=Y_1D_{p_1}+Y_2D_{p_2}$ with $Y_1,Y_2 \in \mathbb{N}[C_m]$.
        \item \cite[Theorem 4.8]{lam2000vanishing} If $k=\prod_{i=1}^sp_i$ with $s \geq 3$ and $p_1 < p_2 < \cdots < p_s$. Assume that $|\rm{supp}(X)| \leq (p_1-1)(p_2-1)+p_3-1$ then $X=\sum_{i=1}^s Y_iD_{p_i}$ with $Y_i \in \mathbb{N}[C_m]$.
        \item\cite[Theorem 3.3]{lam2000vanishing} If $X$ is a minimal $v$-sum and $k=p$ then $X=D_pY$ with $Y\in C_m$.
        \item \cite[Theorem 4.8]{lam2000vanishing} If $X$ is a minimal $v$-sum and $k=\prod_{i=1}^sp_i$ with $s\geq 2$ and $p_1 < p_2 < \cdots < p_s$, then $s\geq 3$ and $ \lVert X \rVert \geq (p_1-1)(p_2-1)+(p_3-1)$.
        
    \end{enumerate}
\end{prop}
%%%%%%%%%%%%%%%%%%%%%%%%%%%%%%%%%%%%%%%%%%%%%%%%%%%%%%%%%%%%%%%%%%%%%%%%%%
%%%%%%%%%%%%%%%%%%%%%%%%%%%%%%%%%%%%%%%%%%%%%%%%%%%%%%%%%%%%%%%%%%%%%%%%%%%%%%%%%%%%%%%%%%%%%%%%%%%%%%%%%%%%%%%%%%%%%%%%%%%%%%%%%%%%%%%%%%%%%%%%%%%%%%%%%%%%%%%%%%%%%%%%%%%%%%%%%%%%%%%%%%%%%%%%%%%%%%%%%%%%%%%%%%%%%%%%%%%%%%%%%%%%%%%%%%%%%%%%%%%%%%%%%%%%%%%%%%%%%%%%%%%%%%%%%%%%%%%%%%%

%\begin{theorem}\cite{kumar1985generalized}
    %Let $q$ and $n$ be positive integers. If $n$ is even  or $q\not\equiv 2 \mod 4$, then there exist a GBF: $\mathbb{Z}_q^n \rightarrow \mathbb{Z}_q$. 
%\end{theorem

%%%%%%%%%%%%%%%%%%%%%%%%%%%%%%%%%%%%%%%%%%%%%%%%%%%%%%%%%%%%%%%%%%%%%%%%%%
\subsection{Existence results for GBFs}

We conclude the preliminaries with two auxiliary existence results and explain why the main focus of the paper is the case in which $m$ is not divisible by 3.

\begin{prop} \label{existence_of_the_gbf_to_lower_order}
Let $f : G \to \mathbb{Z}_m$ be a GBF, and write $m = \prod_{i=1}^t p_i^{\alpha_i}$, where the $p_i$ are distinct primes. Let $k_x$ denote the $c$-exponent of $E_x$ for each $x \in G \setminus \{1_G\}$. Suppose $p$ is a prime such that $p \nmid k_x$ for all $x \in G \setminus \{1_G\}$. Then there exists a GBF $ g: G \to \mathbb{Z}_{m/p}$.
\end{prop}
\begin{proof}
    Proof of the following proposition is similar to that of Proposition~4 in \cite{leung2020new}.
\end{proof}
\begin{prop} \label{gbf_in_lm}
Suppose there exists a GBF $f : G \to \mathbb{Z}_m$. Then, for any positive integer $l$, there also exists a GBF $g : G \to \mathbb{Z}_{lm}$.
\end{prop}

\begin{proof}
Define $g : G \to \mathbb{Z}_{lm}$ by $g(x) = l f(x)$. It is straightforward to verify that $g$ is a GBF.
\end{proof}

\noindent
We now specialise to the case $G = \mathbb{Z}_3^n$. In this paper, we study GBFs from $\mathbb{Z}_3^n$ to $\mathbb{Z}_m$.

Since $3 \not\equiv 2 \pmod{4}$, it follows from Theorem~\ref{P.V.Kumar_existence_resut} that there exists a GBF $f : \mathbb{Z}_3^n \to \mathbb{Z}_3$. Combining this with Proposition~\ref{gbf_in_lm}, we obtain the following result.

\begin{theorem}
There exists a GBF $f : \mathbb{Z}_3^n \to \mathbb{Z}_m$ whenever $m$ is a multiple of $3$.
\end{theorem}
The above theorem establishes the existence of GBFs when $m$ is divisible by $3$. It is  natural to ask the existence of  GBFs  when $m$ is not divisible by $3$.

\noindent
In this paper, we establish several nonexistence results for GBFs. In particular, we investigate the case when $m$ is odd and not divisible by $3$, and the case when $m$ is even and not divisible by $3$.
%%%%%%%%%%%%%%%%%%%%%%%%%%%%%%%%%%%%%%%%%%%%%%%%%%%%%%%%%%%%%%%%%%%%%%%%%%%%%%%%%%%%%%%%%%%%%%%%%%%%%%%%%%%%%%%%%%%%%%%%%%%%%%%%%%%%%%%%%%%%%%%%%%%%%%%%%%%%%%%%%%%%%%%%%%%%%%%%%%%%%
\section{Nonexistence results when $m$ is odd and not divisible by $3$}
Write $m=\prod_{j=1}^tp_j^{\alpha_j}$, where the primes satisfy $5 \leq p_1 < p_2 < \cdots < p_t $ and $\alpha_j$ are positive integers.  Let $\mathcal{P}(k)$ denote the set of all prime divisors of $k$. With these notations, we have the following theorem.

\begin{theorem} \label{Nonexistence_of _gbf_for_odd_m}
    Let there exists a GBF $f: \mathbb{Z}_3^n \rightarrow \mathbb{Z}_m$, then $t \geq 2$ and $2p_1+p_2 \leq 3^n$.
\end{theorem}
\begin{proof}
    Suppose that $f~:~ \mathbb{Z}_3^n \rightarrow \mathbb{Z}_m$ is a GBF. From Lemma 
    \ref{E_x_is_a_vanishing_sum}, it follows that $E_x$ is a $v$-sum for every $x \in \mathbb{Z}_3^n \setminus \{1_G\}$. Now, we consider the following cases.

    \medskip
    \noindent \textbf{Case 1: $t=1$.}\\
     By Lemma \ref{C_expo_is_square_free}, if $t=1$, then the $c$-exponent of $E_x$ will be $p_1$. From Proposition \ref{Discription_of_v-sum} $(ii)$, we have $E_x=D_{p_1}Y$ for some $Y\in \mathbb{N}[C_m]$. Therefore, $3^n=\lVert E_x \rVert = p_1 \lVert y \rVert$, which is impossible since $p_1$ does not divide $3^n$. Therefore, the case $t=1$ cannot occur.
     
    \medskip
    \noindent \textbf{ Case 2: $t \geq 2$.}\\
    As $E_x \in \mathbb{N}[C_m]$ is a $v$-sum, it can be written as a sum of minimal $v$-sums. Let $E_x=\sum_{i=1}^{w} X_i$, where each $X_i$ is a minimal $v$-sum with reduced exponent $k_i$. If $|\mathcal{P}(k_i)| \geq 4$ for some $1\leq i\leq w$, then by Proposition \ref{Discription_of_v-sum}(i), we obtain
    \[
        3^n \geq \|X_i\|
        \geq 2+(p_1-2)+(p_2-2)+(p_3-2)+(p_4-2)
        \geq 3p_1+p_2
        \geq 2p_1+p_2.
    \]
    Thus, the desired inequality holds. 

    \noindent Now assume that $|\mathcal{P}(k_i)| \leq 3$ for all $1 \leq i \leq w$. By Proposition \ref{Discription_of_v-sum}(vi), $|\mathcal{P}(k_i)|$ is either 1 or 3. 

    \medskip
    \noindent \textbf{Case 2.1:} $|\mathcal{P}(k_i)|=3$  for some $1\leq i \leq w$. \\
    From Proposition $\ref{Discription_of_v-sum}$(vi), it follows that 
    \begin{align*}
        \lVert X_i \rVert &\geq ~ (p_1-1)(p_2-1)+(p_3-1) ,\\
        3^n &\geq 4(p_2-1)+p_3 -1\\
        &\geq 3p_2+p_1 ~\geq ~2p_2+p_1\geq 2p_1+p_2.
    \end{align*}
    Hence, the desired inequality follows.\\
    It therefore remains to consider the case in which $|\mathcal{P}(k_i)|=1$ for all $i$.

    \medskip
    \noindent \textbf{Case 2.2:} $|\mathcal{P}(k_i)|=1$ for all $i$.\\
    In this case, each $k_i$ is a prime number, say $k_i=q_i$. From Proposition \ref{Discription_of_v-sum}$(v)$, it follows that $X_i=D_{q_i}Y_i$ for some $Y_i \in C_m$. Thus, \[E_x =\sum_{i=1}^w D_{q_i}Y_i.\] If all $q_i$ are equal, then $E_x=D_{q_1}Y$ for some $Y\in \mathbb{N}[C_m]$, which is impossible since $q_1\nmid 3^n$. Therefore, at least two of the primes $q_i$ are distinct. Without loss of generality, assume that $q_1 \neq q_2$. Then \[E_x =D_{q_1}Y_1+D_{q_2}Y_2+\sum_{i=3}^wD_{q_i}Y_i.\] Therefore, $3^n \geq q_1+q_2+(w-2)p_1$.\\
    If $w \geq 3$, then clearly $3^n \geq q_1+q_2+p_1 \geq 2p_1+p_2$. Furthermore, if $w=2$, then $E_x = D_{q_1}Y_1+D_{q_2}Y_2$, which implies $3^n=q_1+q_2$. This is impossible because $q_1+q_2$ is an even number, whereas $3^n$ is odd. Hence, the desired result follows.
\end{proof}

\begin{corollary} \label{non_exist_for_n_1_2_and_m_odd}
    Let $m$ be an odd integer not divisible by $3$. Then no generalized bent function $f:\mathbb{Z}_3^n \to \mathbb{Z}_m$ exists for $n=1$ or $n=2$.
\end{corollary}

\noindent  We now specialise the general discussion to the case $n=3$.
 
\subsection*{Nonexistence results when $n=3$} Now, we use Proposition \ref{Discription_of_v-sum} to determine the possible forms of the autocorrelation terms $E_x$ in the case $n=3$. Throughout the remainder of this section, we set $G=\mathbb{Z}_3^3$.
\begin{theorem}\label{possible_form_of_e_x_n_3}
    Let $f: \mathbb{Z}_3^3 \rightarrow \mathbb{Z}_m$ be a GBF, where $m = \prod_{j=1}^tp_j^{\alpha_j}$ with primes satisfying $5\leq p_1 < p_2 \cdots < p_t$ and $\alpha_j$ are positive integers. Then, for each $x\in G\setminus {1_G}$, $E_x$ is one of the following forms:
    \begin{enumerate}[label=(\roman*)]
    \item $E_x = D_5\bigl(g^{\alpha_x} + g^{\beta_x} + g^{\gamma_x} + g^{\delta_x}\bigr) + D_7 g^{\omega_x},$
    \item $E_x = D_5 g^{\alpha_x} + D_{11}\bigl(g^{\beta_x} + g^{\gamma_x}\bigr),$
    \item $E_x = D_5\bigl(g^{\alpha_x} + g^{\beta_x}\bigr) + D_{17} g^{\gamma_x},$
    \item $E_x = D_7\bigl(g^{\alpha_x} + g^{\beta_x}\bigr) + D_{13} g^{\gamma_x},$
    
    \end{enumerate}
     where $\alpha_x, \beta_x, \gamma_x, \delta_x,$ and $\omega_x$ are integers.
\end{theorem}
\begin{proof} 
Suppose $f: G \rightarrow \mathbb{Z}_m$ is GBF. Then, for each $x \in G \setminus \{1_G\}$, the element $E_x$ is a minimal $v$-sum. Let $k_x$ denote the $c$-exponent of $E_x$. By Lemma \ref{C_expo_is_square_free}, we can write $k_x =\prod \limits_{i=1}^rq_i$, where $q_1 < q_2 \cdots <q_r$ are distinct primes and $\{q_1,q_2,\ldots, q_r\} \subseteq \{p_1,p_2, \ldots ,p_t\} $.

\medskip
\noindent \textbf{Case 1:} $r \geq 3$.\\
Note that 
\[ |\operatorname{supp}(E_x)|=27 \leq (5-1)(7-1)+(10-1) \leq (q_1-1)(q_2-1)+q_3-1.
\]
By Proposition \ref{Discription_of_v-sum} (iv), we have $E_x= \sum_{i=1}^rY_iD_{q_i}$, where $Y_i \in \mathbb{N}[C_m]$. Clearly, $\lVert E_x \rVert =\sum_{i=1}^rq_i \lVert Y_i \rVert $. If $r \geq 4$, then $27 = \sum_{i=1}^rq_i\lVert Y_i \rVert \geq 5+7+11+13$, which is not possible. Hence, we may assume that $r=3$. In this case, $27=\sum_{i=1}^3q_in_i$, where $n_i =\lVert Y_i \rVert \in \mathbb{N}$. Considering the smallest possible primes, we have 
\begin{align*}
    5+7+11 & \leq 23,\\
    10+7+11 & \geq 27,\\
    7+11+13& \geq 27.
\end{align*} 
Thus, there is no possible choice of $(q_1,q_2,q_3)$ such that $27=\sum_{i=1}^3q_in_i$ holds. Hence, the case $r\geq 3$ is not possible. Therefore, we must have $r \leq 2$.

\medskip
\noindent \textbf{Case 2:} $r=1$.\\ 
In this case, by Proposition \ref{Discription_of_v-sum}$(v)$, we have $E_x= D_qY$, and hence $27=q \lVert Y \rVert$. This is not possible since $q \geq 5$.

\medskip
\noindent \textbf{Case 3:} $r=2$.\\ 
By Proposition \ref{Discription_of_v-sum}$(iii)$, we have 
$E_x=D_{q_1}Y_1+D_{q_2}Y_2$. Therefore,
\begin{align} \label{Eq^n_btw_possible_value_of_$E_x$}
    27=\lVert E_x \rVert =q_1 \lVert Y_1 \rVert +q_2 \lVert Y_2 \rVert.
\end{align}
If $ \lVert Y_1 \rVert = \lVert Y_2 \rVert = 1$, then $27=q_1+q_2$, which is not possible since $q_1$ and $q_2$ are odd primes. Therefore, at least one of them is greater than $1$. The only possible values of $(q_1,q_2)$ for which Equation \eqref{Eq^n_btw_possible_value_of_$E_x$} holds are $(5,7),(5,11),(5,17),~\text{and}~(7,13)$. This leads to the possible forms of $E_x$. 
\end{proof}

 \begin{remark}
     For any non identity element $x \in \mathbb{Z}_3^3$, there exists a subgroup $H$ which is isomorphic to $\mathbb{Z}_3^2$ and $H, Hx, Hx^2$ are the distinct cosets of $H$ in $\mathbb{Z}_3^3$. Write $x=(x_1,x_2,x_3)\in \mathbb{Z}_3^3$. Since $x$ is a non identity element, at least one coordinate of $x$ is nonzero. Without loss of generality, assume that $x_1\neq 0$. Define $H=\{0\}\times\mathbb{Z}_3 \times \mathbb{Z}_3$. Clearly, $H$ is isomorphic to $\mathbb{Z}_3^2$. Moreover, $x\notin H$ since $x_1\neq 0$. Therefore, $Hx$ and $Hx^2$ are two distinct cosets of $H$ in $\mathbb{Z}_3^3$.

% \noindent  {\color{ red} What do you mean by $x^2$ in $\mathbb{Z}_3^3$. Let $x=(1,0,0)$, $x^2=?$}
 \end{remark}

Let  $x$ be a non identity element in $\mathbb{Z}_3^3$ and $H=\{e_1,e_2,\ldots,e_9\}$ be a subgroup of $\mathbb{Z}_3^3$ such that
 \begin{enumerate}[label=(\roman*)]
     \item $H$ is isomorphic to $\mathbb{Z}_3^2$.
     \item $Hx$ and $Hx^2$ are two distinct cosets of $H$ in $\mathbb{Z}_3^3$. 
 \end{enumerate}
Consider a map $f:\mathbb{Z}_3^3\rightarrow \mathbb{Z}_m$. Using the decomposition of $\mathbb{Z}_3^3$ into subgroup $H$ and the cosets $Hx$ and $Hx^2$, we can write
\begin{align} \label{expression _of_D_f_as_elemets_of cosets}
    D_f =\sum_{i=1}^9 g_i(e_i+h_ie_ix+w_ie_ix^2),
\end{align}
where 
\begin{align*}
    g_i&=g^{f(e_i)},\\
    h_i &= g^{f(e_ix)-f(e_i)},\\
    w_i &= g^{f(e_ix^2)-f(e_i)}.
\end{align*}
Note that 
\begin{align} \label{D_fD_f^-1_in_terms_of_oth_exp_of_D_f}
    D_fD_f^{(-1)} &= \left (\sum_{i=1}^9 g_i(e_i+h_ie_ix+w_ie_ix^2)\right)\left(\sum_{j=1}^9 g_j^{-1}(e_j^{-1}+h_j^{-1}e_j^{-1}x^2+w_j^{-1}e_j^{-1}x)\right),\\
    &= \sum_{i,j}g_ig_j^{-1}(1+h_ih_j^{-1}+w_iw_j^{-1})e_ie_j^{-1} \nonumber\\
    &+\sum_{i,j}g_ig_j^{-1}(h_i+w_j^{-1}+w_ih_j^{-1})e_ie_j^{-1}x \nonumber\\
    &+\sum_{i,j}g_ig_j^{-1}(h_j^{-1}+w_i+h_iw_j^{-1})e_ie_j^{-1}x^2. \nonumber
\end{align}
Comparing Equations (\ref{D_fD_f-1 =_sum_of E_x}) and (\ref{D_fD_f^-1_in_terms_of_oth_exp_of_D_f}), for each $1\leq k \leq 9$, we obtain the following expressions
\begin{align}
    E_{e_k} &= \sum_{\substack{i,j\\e_ie_j^{-1}=e_k}}g_ig_j^{-1}\bigl(1+h_ih_j^{-1}+w_iw_j^{-1}\bigr),\nonumber\\
    E_{e_kx}&=\sum_{\substack{i,j\\e_ie_j^{-1}=e_k}}g_ig_j^{-1}(h_i+w_j^{-1}+w_ih_j^{-1}), \nonumber\\
    E_{e_kx^2} &=\sum_{\substack{i,j\\e_ie_j^{-1}=e_k}}g_ig_j^{-1}(h_j^{-1}+w_i+h_iw_j^{-1}).\nonumber
\end{align}
In particular, for the coset representative $x$, we have 
\begin{align}\label{exp_of_E_x}
    E_x=\sum_{i=1}^9 h_i+w_i^{-1}+w_ih_i^{-1}.
\end{align}
We now define two group homomorphisms that will play an important role in establishing some nonexistence results for GBFs.
\begin{enumerate}[label=(\roman*)]

    \item Let
        \(
        m=\prod_i p_i^{a_i}\)
        and \( m'=\prod_i p_i^{b_i}\),
        where $a_i>0$, $b_i\geq 0$, and $b_i\leq a_i$ for each $i$.
        Each element of $C_m$ can be written uniquely as $\prod_i g_i$, where
        $\operatorname{ord}(g_i)=p_i^{a_i}$. Define a homomorphism
        \[
        \phi_{m'}:C_m \rightarrow C_{m/m'}
        \]
        by
        \begin{align}\label{phi map}
        \phi_{m'}\left(\prod_i g_i\right)=\prod_i g_i^{p_i^{b_i}}.
        \end{align}
    \item For $x \in G \setminus{1_G}$, define $\tau_x:\mathbb{Z}[C_m\cdot G] \rightarrow \mathbb{Z}[C_m \cdot H][\omega]$ a homomorphism such that 
    \begin{align}
        \tau_x{(e_i)} &=e_i \qquad \text{for all } 1 \leq i \leq 9, \nonumber \\
        \tau_x{(x)} & = \omega, \\
        \tau_x{(g)}& =g  \qquad \text{for} \hspace{2mm}g \in C_m, \nonumber
    \end{align}
where $H=\{e_1,e_2,\ldots,e_9\}$ is a subgroup of $\mathbb{Z}_3^3$ that is isomorphic to $\mathbb{Z}_3^2$ so that $Hx$ and $Hx^2$ are two other cosets of $H$ in $\mathbb{Z}_3^3$.  
\end{enumerate}
Equation \eqref{exp_of_E_x} provides an alternative representation of $E_x$, which allows us to express $D_fD_f^{(-1)}$ in two different ways. By applying these homomorphisms to $D_fD_f^{(-1)}$ and comparing the resulting expressions, we obtain contradictions involving the coefficients of $E_x$. In the remainder of this section, we will consider $m=5m'$ and $5\nmid m'$.

\subsubsection*{Nonexistence results when $n=3, m=5m'$ and $5\nmid m'$ }
\begin{lema}\label{special_form_of_E_x on action of phi_m}
Let $\phi_{m'}:C_{m}  \rightarrow C_5$ be the group homomorphism defined in (\ref{phi map}), where $m=5m'$ and $5\nmid m'$. Suppose that there exists $x \in G\setminus\{1_G\}$ such that $E_x= D_5g^{\alpha_x}+D_{11}\bigl(g^{\beta_x}+g^{\gamma_x}\bigr)$. Then 
\begin{align*}
    \phi_{m'}(E_x)=22+D_5=23+u_1+u_1^{-1}+u_2+u_2^{-1},
\end{align*}
where $u_i \in C_5$.  
\end{lema}
\begin{proof}
 Let $\phi_{m'}:C_{5m'}  \rightarrow C_5$ be the group homomorphism defined in (\ref{phi map}). Suppose that there exists  $x \in G\setminus\{1_G\}$ such that $E_x= D_5g^{\alpha_x}+D_{11}(g^{\beta_x}+g^{\gamma_x})$, then
 \begin{align} \label{first_expression_of_phi_E_x}
     \phi_{m'}(E_x)=D_5+11(g_5^{\beta_x}+g_5^{\gamma_{x}}),
 \end{align}
 where $g_5$ is a generator of $C_5$.
 For this fixed $x \in G$, let $H=\{e_1,e_2,\ldots,e_9\}$ be a subgroup of $\mathbb{Z}_3^3$ isomorphic to $\mathbb{Z}_3^2$ such that $Hx$ and $Hx^2$ are two other cosets of $H$ in $\mathbb{Z}_3^3$. From Equation (\ref{exp_of_E_x}), we have $E_x= \sum_{i=1}^9h_i+w_i^{-1}+h_iw_i^{-1}$. Therefore, 
 \begin{align} \label{second_exp_of_phi_E_x}
   \phi_{m'}(E_x)= \sum_{i=1}^9 \Bigl( \phi_{m'}(h_i)+\phi_{m'}(w_i^{-1})+\phi_{m'}(h_i^{-1}w_i)\Bigr).
 \end{align}
 We can arrange the elements involved in Equation (\ref{second_exp_of_phi_E_x}) into $9$ layers, each consisting of $3$ elements. The $i^{th}$ layer is given by
\[
\bigl(
\phi_{m'}(h_i), \ \phi_{m'}(w_i^{-1}), \ \phi_{m'}(h_i^{-1} w_i)
\bigr).
\]
Observe that, within each layer, any one element is the inverse of product of other two. On the other hand, from Equation (\ref{first_expression_of_phi_E_x}), we can say that, apart from the contribution coming from $D_5$, the elements $g_5^{\beta_x}$ and $g_5^{\gamma_x}$ each appear exactly $11$ times. The remaining five elements are precisely those of $D_5$, namely $\{1, u_1, u_1^{-1}, u_2, u_2^{-1}\}$, where $u_i \in C_5$. Note that $g_5^{\beta_x}$, $g_5^{\gamma_x}\in \{1, u_1, u_1^{-1}, u_2, u_2^{-1}\} $.    Since each layer satisfies the property that one element is the inverse of the product of the other two, it follows that if a layer contains $1$ and an element $\alpha \in C_5$, it must also contain $\alpha^{-1}$.  
Without loss of generality, we may assume that $\{1, u_1, u_1^{-1}\}$ appears in one layer. Consequently, the elements of $D_5$ can occupy at most three layers. Therefore, at least six layers are entirely filled with the elements $g_5^{\beta_x}$ and $g_5^{\gamma_x}$. We will see the distribution of $g_5^{\beta_x}$ and $g_5^{\gamma_x}$ in these six layers. We now consider the following cases. 

\medskip
 \noindent \textbf{Case 1}: $\beta_x \equiv \gamma_x \pmod{5}$. \\
 Since any element in each row is the inverse of the product of two other elements. Therefore, this case will hold only when $(g_5^{\beta_x+\beta_x} )^{-1}=g_5^{\beta_x}$, which implies $-2\beta_x \equiv \beta_x \pmod{5}$, that is, $-3\beta_x \equiv 0 \pmod{5}$. This is possible only when $\beta_x \equiv 0 \pmod{5}$. Therefore, this case occurs only when $g_5^{\beta_x} = g_5^{\gamma_x}=1$. In particular, this case also shows that all three elements in a layer will be equal only if they are identity.

 \medskip
 \noindent \textbf{Case 2:} $\beta_x \not \equiv \gamma_x \pmod{5}$.\\
 In this case, $g_5^{\beta_x}$ and $g_5^{\gamma_x}$ are distinct elements of $C_5$. If $\gamma_x \not\equiv 0 \pmod{5}$, then from Case~1, a layer of the form 
$
(g_5^{\gamma_x}, g_5^{\gamma_x}, g_5^{\gamma_x})
$
cannot occur. Since $\beta_x \not\equiv \gamma_x \pmod{5}$, at least one of $\beta_x$ or $\gamma_x$ is nonzero modulo $5$. Therefore, in this case, each of the six layers consists of two distinct elements among $g_5^{\beta_x}$ and $g_5^{\gamma_x}$. Moreover,
$
(g_5^{\beta_x} g_5^{\gamma_x})^{-1}
$
must be equal to either $g_5^{\beta_x}$ or $g_5^{\gamma_x}$.

 \medskip
 \noindent \textbf{case 2.1} $\beta_x \equiv 0 \pmod{5}$ and $\gamma_x \not \equiv 0 \pmod{5}$.\\
Clearly, this case will occur only when $(g_5^{\gamma_x}\cdot g_5^{0})^{-1}$ is either $g_5^{0}$ or $g_5^{\gamma_x}$.\\
If $(g_5^{\gamma_x} \cdot g_5^0)^{-1} = g_5^0$, then $\gamma_x \equiv 0 \pmod{5}$, which is a contradiction. \\
If $(g_5^{\gamma_x} \cdot g_5^0)^{-1} = g_5^{\gamma_x}$, then $g_5^{-\gamma_x} = g_5^{\gamma_x}$, which implies $\gamma_x \equiv 0 \pmod{5}$, again a contradiction. \\
Hence, this case is not possible.

\medskip
\noindent \textbf{case 2.2:} $\beta_x \equiv -\gamma_x \pmod{5} $.\\
Here $(g_5^{\beta_x} g_5^{\gamma_x})^{-1} = g_5^{0}$, which is not possible since $\beta_x, \gamma_x \not\equiv 0 \pmod{5}$. Hence, this case will not occur.\\
Now consider the remaining following cases.

\medskip
\noindent \textbf{Case 2.3} $\beta_x \equiv 1 \pmod{5}$ and $\gamma_x \equiv 2 \pmod 5$. 

\noindent\textbf{Case 2.4} $\beta_x \equiv 1 \pmod{5}$ and $\gamma_x \equiv 3 \pmod {5}$.

\noindent \textbf{Case 2.5} $\beta_x \equiv 2 \pmod{5}$ and $\gamma_x \equiv 4 \pmod {5}$.

\noindent \textbf{Case 2.6} $\beta_x \equiv 3 \pmod{5}$ and $\gamma_x \equiv 4 \pmod {5}$.\\
In each of these cases, $(g_5^{\beta_x} g_5^{\gamma_x})^{-1}$ is equal to either $g_5^{\beta_x}$ or $g_5^{\gamma_x}$. As each layer contains two distinct elements, therefore, if $(g_5^{\beta_x} g_5^{\gamma_x})^{-1} = g_5^{\beta_x}$, then each such layer is of the form
\[
(g_5^{\beta_x}, g_5^{\beta_x}, g_5^{\gamma_x}).
\]
Similarly, if 
$(g_5^{\beta_x} g_5^{\gamma_x})^{-1} = g_5^{\gamma_x}$,
then each such layer is of the form
\[(g_5^{\beta_x}, g_5^{\gamma_x}, g_5^{\gamma_x}).\]
In either case, one of the elements $g_5^{\beta_x}$ or $g_5^{\gamma_x}$ appears at least $12$ times in these six layers. This contradicts the fact that, apart from the contribution coming from $D_5$, each of $g_5^{\beta_x}$ and $g_5^{\gamma_x}$ appears exactly $11$ times.\\ 
Therefore, the only possible value for $g_5^{\beta_x}$ and $g_5^{\gamma_x}$ is $1$. Hence $\phi_{m'}(E_x)=22+D_5=23+u_1+u_1^{-1}+u_2+u_2^{-1}$.
\end{proof}

\begin{prop} \label{Result_phi_of_diff_of_vanishing_sum}
  Let $\phi_{m'}:C_{m}  \rightarrow C_5$ be the group homomorphism defined in (\ref{phi map}), where $m=5m'$ and $5\nmid m'$. For $x\in G\setminus{\{1_G\}}$, let $H=\{e_1,e_2,\ldots,e_9\}$ be a subgroup of $G$ of order $9$ with $e_1=1_G$. From Equation (\ref{expression _of_D_f_as_elemets_of cosets}), we may write $D_f = \sum_{i=1}^9 g_i(1+h_ix+w_ix^2)$. Suppose there exists $x \in G\setminus{\{1_G\}}$ such that 
  \begin{align} \label{simplified_exp_for_phi_m'}
      \phi_{m'}(E_x)=22+D_5=23+u_1+u_1^{-1}+u_2+u_2^{-1},
  \end{align} where $u_1,u_2 \in C_5$. Then, upto relabeling of elements, one of the following holds: 
  \begin{enumerate}[label=(\roman*)]
      \item 
        \begin{align*}
            \phi_{m'}(h_1) &=u_1, \text ~{and}~ \phi_{m'}(w_1) \in \{1, u_1\}, \\
            \phi_{m'}(h_2) &= u_2, ~\text{and} ~\phi_{m'}(w_2)\in \{1,u_2\},\\
            \phi_{m'}(h_i) & =\phi_{m'}(w_i) = 1~ ~ ~ \text{ for all } ~~~3 \leq i \leq 9. 
        \end{align*}
      \item 
        \begin{align*}
           \phi_{m'}(h_1) &=u_1, \text ~{and}~ \phi_{m'}(w_1) \in \{1,u_1\}, \\
           \phi_{m'}(h_2) & \in \{1,u_2\}, \text ~{and}~ \phi_{m'}(w_2)= u_2,\\
           \phi_{m'}(h_i) & =\phi_{m'}(w_i) = 1~ ~ ~ \text{ for all } ~~~3 \leq i \leq 9. 
        \end{align*}
\end{enumerate}
 Moreover, the following identities hold:
  \begin{align*}
      \phi_{m'}(E_{e_1}-E_{e_1x^2})& = 4-u_1-u_1^{-1}-u_2-u_2^{-1},\\
      \phi_{m'}(E_{e_1x}-E_{e_1x^2})& =0,\\
      \phi_{m'}(E_{e_2}-E_{e_2x^2})& = 
      \begin{cases}
          \phi_{m'}(g_2)(1-u_1^{-1}-u_2+u_2u_1^{-1}),\\
          0,
      \end{cases}\\
     \phi_{m'}(E_{e_2x}-E_{e_2x^2})& =
     \begin{cases}
         \phi_{m'}(g_2)\left ( 1-u_1^{-1}-u_2+u_2u_1^{-1} \right ),\\
         0.
     \end{cases}
  \end{align*}
\end{prop}

\begin{proof}
   From Equation (\ref{exp_of_E_x}), we have $E_x=\sum_{i=1}^9 (h_i+w_i^{-1}+h_i^{-1}w_i)$. Therefore,
   \begin{align} \label{phi_m'_in_terms_of_h_w_wh}
       \phi_{m'}(E_x)=\sum_{i=1}^9 \phi_{m'}(h_i)+\phi_{m'}(w_i^{-1})+\phi_{m'}(h_i^{-1}w_i).
   \end{align}
Comparing Equations (\ref{simplified_exp_for_phi_m'}) and (\ref{phi_m'_in_terms_of_h_w_wh}), there exist indices $k$ and $j$ such that
\begin{align*}
    \{\phi_{m'}(h_k) ,\phi_{m'}(w_k^{-1}), \phi_{m'}(h_k^{-1}w_k) \} &=\{1,u_1,u_1^{-1}\},\\
    \{\phi_{m'}(h_j) ,\phi_{m'}(w_j^{-1}), \phi_{m'}(h_j^{-1}w_j) \} &=\{1,u_2,u_2^{-1}\},\\
    \{\phi_{m'}(h_i) ,\phi_{m'}(w_i^{-1}), \phi_{m'}(h_i^{-1}w_i) \} &=\{1\} ~~~ \text{for all}~~1\leq i\leq 9 ~\text{and} ~i\neq j,k.
\end{align*}
By replacing $D_f$ with a suitable translate, we may assume that $g_1=1_m$ and $\phi_{m'}(h_1)=u_1$. Consequently, $\phi_{m'}(w_1)$ is either $1$ or $u_1$. After relabeling $e_2,\ldots,e_9$, we may assume that 
\begin{align*}
    \{\phi_{m'}(h_2) ,\phi_{m'}(w_2^{-1}), \phi_{m'}(h_2^{-1}w_2) \} &=\{1,u_2,u_2^{-1}\}.
\end{align*}
Therefore, we have the following cases. 
\begin{enumerate}[label=(\roman*)]
\item[\textbf{Case 1:}] \begin{align*}
            \phi_{m'}(h_1) &=\phi_{m'}(w_1) = u_1, \\
            \phi_{m'}(h_2) &=\phi_{m'}(w_2)=u_2,\\
            \phi_{m'}(h_i) & =\phi_{m'}(w_i) = 1~ ~ ~ \text{ for all } ~~~3 \leq i \leq 9. 
        \end{align*}
\item[\textbf{Case 2:}] \begin{align*}
           \phi_{m'}(h_1) &=u_1,~\phi_{m'}(w_1) = u_1, \\
           \phi_{m'}(h_2) &=u_2,~\phi_{m'}(w_2)= 1,\\
           \phi_{m'}(h_i) & =\phi_{m'}(w_i) = 1~ ~ ~ \text{ for all } ~~~3 \leq i \leq 9. 
        \end{align*}
\item[\textbf{Case 3:}]  \begin{align*}
           \phi_{m'}(h_1) &=u_1,~\phi_{m'}(w_1) = u_1, \\
           \phi_{m'}(h_2) &=1,~\phi_{m'}(w_2)= u_2,\\
           \phi_{m'}(h_i) & =\phi_{m'}(w_i) = 1~ ~ ~ \text{ for all } ~~~3 \leq i \leq 9. 
        \end{align*}
\item[\textbf{Case 4:}]  \begin{align*}
           \phi_{m'}(h_1) &=u_1,~\phi_{m'}(w_1) = 1, \\
           \phi_{m'}(h_2) &=u_2,~\phi_{m'}(w_2)= u_2,\\
           \phi_{m'}(h_i) & =\phi_{m'}(w_i) = 1~ ~ ~ \text{ for all } ~~~3 \leq i \leq 9. 
        \end{align*}
\item[\textbf{Case 5:}]  \begin{align*}
           \phi_{m'}(h_1) &=u_1,~\phi_{m'}(w_1) = 1, \\
           \phi_{m'}(h_2) &=u_2,~\phi_{m'}(w_2)= 1,\\
           \phi_{m'}(h_i) & =\phi_{m'}(w_i) = 1~ ~ ~ \text{ for all } ~~~3 \leq i \leq 9. 
        \end{align*}
\item[\textbf{Case 6:}] \begin{align*}
           \phi_{m'}(h_1) &=u_1,~\phi_{m'}(w_1) = 1, \\
           \phi_{m'}(h_2) &=1,~\phi_{m'}(w_2)= u_2,\\
           \phi_{m'}(h_i) & =\phi_{m'}(w_i) = 1~ ~ ~ \text{ for all } ~~~3 \leq i \leq 9. 
        \end{align*}
\end{enumerate}

We carry out the computation for Case 1 and the computations for the remaining cases are analogous.
From Equation (\ref{expression _of_D_f_as_elemets_of cosets}),
\begin{align*}
    \phi_{m'}(D_f) = \sum_{i=1}^9 \phi_{m'}(g_i)\bigl(1+\phi_{m'}(h_i)x+\phi_{m'}(w_i)x^{2}\bigr)e_i.
\end{align*}
Therefore, in case 1,
\begin{align*}
    \phi_{m'}(D_f) = (1+u_1x+u_1x^2)e_1+\phi_{m'}(g_2)(1+u_2x+u_2x^2)e_2+ \sum_{i=3}^9 \phi_{m'}(g_i)(1+x+x^{2})e_i.
\end{align*}
Applying $\tau_x$, we obtain
\begin{align*}
    \tau_x\bigl(\phi_{m'}(D_fD_f^{(-1)})\bigr)= &\left ((1+u_1\omega+u_1\omega^2)e_1+ \phi_{m'}(g_2)(1+u_2\omega +u_2\omega^2)e_2 \right )\\ &\left (  (1+u_1^{-1}\omega^2+u_1^{-1}\omega)e_1 +  \phi_{m'}(g_2^{-1})(1+u_2^{-1}\omega^2 +u_2^{-1}\omega)e_2^{-1} \right ).
\end{align*}
The coefficient of $e_1$ in $ \tau_x(\phi_{m'}(D_fD_f^{(-1)}))$ is 
\begin{align*}
    (1+u_1w+u_1w^2)(1+u_1^{-1}w^2+u_1^{-1}w)+\phi_{m'}(g_2)\phi_{m'}(g_2^{-1})(1+u_2w+u_2w^2)(1+u_2^{-1}w^2+u_2^{-1}w).
\end{align*}
Simplifying the above equation, coefficient of $e_1$ is
\begin{align}\label{coefficient_of_e_1_1st exp}
    &\bigl(3+(w+w^2)+u_1^{-1}(w+w^2)+u_1(w+w^2)\bigr)+\bigl(3+(w+w^2)+u_2^{-1}(w+w^2)+u_2(w+w^2)\bigr) \nonumber \\
    &=2-u_1+u_1^{-1}+2-u_2-u_2^{-1} \nonumber\\
    &=4-u_1-u_2-u_1^{-1}-u_2^{-1}.
\end{align}
Similarly, the coefficient of $e_2$ in $ \tau_x(\phi_{m'}(D_fD_f^{-1}))$ is 
\begin{align*}
   \phi_{m'}(g_2)(1+u_2\omega +u_2\omega^2)(1+u_1^{-1}\omega^2+u_1^{-1}\omega).
\end{align*}
Simplifying the above equation, coefficient of $e_2$ is
\begin{align} \label{coefficients_of_e_2_1st exp}
    &\phi_{m'}(g_2)(1+u_1^{-1}(\omega+\omega^{2})+u_2(\omega+\omega^{2})+u_2u_1^{-1}(\omega+\omega^{2})+2u_2u_1^{-1}) \nonumber\\
   &=\phi_{m'}(g_2)(1-u_1^{-1}-u_2-u_2u_1^{-1}+2u_2u_1^{-1}) \nonumber\\
   &=\phi_{m'}(g_2)(1-u_1^{-1}-u_2 + u_2u_1^{-1}).
\end{align}
\noindent From Equation (\ref{D_fD_f-1 =_sum_of E_x}), we have $D_fD_f^{(-1)}=\sum_{i=1}^{9} (E_{e_i}+E_{e_ix}x+ E_{e_ix^2}x^2)e_i$. Therefore, 
\begin{align*}
    \tau_x(\phi_{m'}(D_fD_f^{(-1)}))= \sum_{i=1}^{9} (\phi_{m'}(E_{e_i})+\phi_{m'}(E_{e_ix})\omega+ \phi_{m'}(E_{e_ix^2})\omega^{2})e_i.
\end{align*}
Substituting $\omega^{2}$ by $-1-w$, we have
\begin{align}\label{expression_of_tphi_D_fD_f^-1}
    \tau_x(\phi_{m'}(D_fD_f^{(-1)}))= \sum_{i=1}^{9} ((\phi_{m'}(E_{e_i})-\phi_{m'}(E_{e_ix^2}))+(\phi_{m'}(E_{e_ix})- \phi_{m'}(E_{e_ix^2})\omega)e_i.
\end{align}
From the Equation (\ref{expression_of_tphi_D_fD_f^-1}), coefficient of $e_1$ in $\tau_x(\phi_{m'}(D_fD_f^{-1}))$ is
\begin{align} \label{coefficient_of_e_1_2nd exp}
    \phi_{m'}(E_{e_1}-E_{e_1x^2})+\phi_{m'}(E_{e_1x}-E_{e_1x^2})\omega.
\end{align}
From the Equation (\ref{expression_of_tphi_D_fD_f^-1}), coefficient of $e_2$ in $\tau_x(\phi_{m'}(D_fD_f^{-1}))$ is
\begin{align} \label{coefficient_of_e_2_2nd exp}
    \phi_{m'}(E_{e_2}-E_{e_2x^2})+\phi_{m'}(E_{e_2x}-E_{e_2x^2})\omega.
\end{align}
Coefficients of $e_1$ and $e_2$ in $\tau_x(\phi_{m'}(D_fD_f^{-1}))$ belong to $\mathbb{Z}[C_5][\omega]$. As $\mathbb{Z}[C_5][\omega]$ has a basis $\{1,\omega\}$ over $\mathbb{Z}[C_5]$. Therefore, comparing the coefficient of $e_1$ from Equations (\ref{coefficient_of_e_1_1st exp}) and (\ref{coefficient_of_e_1_2nd exp}), we have 
\begin{align*}
\phi_{m'}(E_{e_1}-E_{e_1x^2}) &= 4-u_1-u_1^{-1}-u_2 -u_2^{-1},\\
\phi_{m'}(E_{e_1x}-E_{e_1x^2})& =0.
\end{align*}
Similarly, upon comparing the coefficient of $e_2$ from Equations (\ref{coefficients_of_e_2_1st exp}) and (\ref{coefficient_of_e_2_2nd exp}), we have 
\begin{align*}
\phi_{m'}(E_{e_2}-E_{e_2x^2}) &= \phi_{m'}(g_2)(1-u_1^{-1}-u_2 + u_2u_1^{-1}),\\
\phi_{m'}(E_{e_2x}-E_{e_2x^2})& =0.
\end{align*}
Note that both $\phi_{m'}(E_{e_2}-E_{e_2x^2}) $ and $\phi_{m'}(E_{e_2x}-E_{e_2x^2})$ are not simultaneously zero. Repeating the same argument for the other cases, we have
\begin{align*}
\phi_{m'}(E_{e_1}-E_{e_1x^2}) &= 4-u_1-u_1^{-1}-u_2 -u_2^{-1},\\
\phi_{m'}(E_{e_1x}-E_{e_1x^2})& =0,\\
    \phi_{m'}(E_{e_2}-E_{e_2x^2})& = 
      \begin{cases}
          \phi_{m'}(g_2)(1-u_1^{-1}-u_2+u_2u_1^{-1}),\\
          0,
      \end{cases}\\
     \phi_{m'}(E_{e_2x}-E_{e_2x^2})& =
     \begin{cases}
         \pm \phi_{m'}(g_2)\left ( 1-u_1^{-1}-u_2+u_2u_1^{-1} \right ),\\
         0.
     \end{cases}
\end{align*}
Moreover, in each case, the two quantities $\phi_{m'}(E_{e_2}-E_{e_2x^2}) $ and $\phi_{m'}(E_{e_2x}-E_{e_2x^2})$ do not vanish simultaneously.
\end{proof}
\noindent The following proposition is a direct consequence of the above Proposition \ref{Result_phi_of_diff_of_vanishing_sum}. 
\begin{prop}\label{coff_of_diff_of_vanishing_sum}
Let $\phi_{m'}:C_{m}  \rightarrow C_5$ be the group homomorphism defined in (\ref{phi map}), where $m=5m'$ and $5\nmid m'$. Suppose there exists $x \in G\setminus\{1_G\}$ such that $\phi_{m'}(E_x)=22+D_5=23+u_1+u_1^{-1}+u_2+u_2^{-1}$. Then at least one of $\phi_{m'}(E_{e_2}-E_{e_2x^2})$ or $\phi_{m'}(E_{e_2x}-E_{e_2x^2})$ is nonzero. Moreover, the possible values of nonzero coefficients are equal to $\pm 1$.
\end{prop}

\begin{prop}\label{resul_excluiding_one_form_of_e_x}
    Let $f: \mathbb{Z}_3^3 \rightarrow \mathbb{Z}_{5m'}$ be a GBF, where $5 \nmid m'$. Then, for each $x\in G\setminus\{1_G\}$, $ E_x \neq D_5g^{\alpha_x}+D_{11}(g^{\beta_x}+g^{\gamma_x})$. Moreover, $11\nmid k_x$ for each $x\in G\setminus{1_G}$.
\end{prop}
\begin{proof}
   Let $\phi_{m'}:C_{5m'}  \rightarrow C_5$ be the group homomorphism defined in (\ref{phi map}). Suppose there exists $x\in G\setminus\{1_G\}$, such that $E_x= D_5g^{\alpha_x}+D_{11}(g^{\beta_x}+g^{\gamma_x})$,  then from Lemma \ref{special_form_of_E_x on action of phi_m} $\phi_{m'}(E_x)= 23+u_1+u_1^{-1}+u_2+u_2^{-1}$. From Proposition \ref{coff_of_diff_of_vanishing_sum}, at least one of $\phi_{m'}(E_{e_2}-E_{e_2x^2})$ or $\phi_{m'}(E_{e_2x}-E_{e_2x^2})$ is nonzero and the nonzero coefficients are equal to $1$ or $-1$.  Without loss of generality, assume that $\phi_{m'}(E_{e_2}-E_{e_2x^2})$ is nonzero.  We consider the following four possible forms of $\phi_{m'}(E_y)$, where $y\in G\setminus\{1_G\}$.
    \begin{enumerate}[label=(\roman*)]
       \item $\phi_{m'}(E_y) =4D_5+7g_5^{\omega_y}$,
       \item  $\phi_{m'}(E_y)=D_5+22$,
       \item  $\phi_{m'}(E_y)=2D_5+17g_5^{t_y}$,
       \item  $\phi_{m'}(E_y)=7(g_5^{\alpha_y}+g_5^{\beta_y})+13g_5^{\gamma_y}$.
    \end{enumerate}
    We compute $\phi_{m'}(E_{e_2}-E_{e_2x^2})$ in all possible cases and show that a coefficient with absolute value greater than $1$ always appears, which yields a contradiction. Firstly, we will consider the cases in which $\phi_{m'}(E_{e_2})$ and $\phi_{m'}(E_{e_2x^2})$ are of different type.

    \medskip
\noindent \textbf{Case 1:} $\phi_{m'}(E_{e_2})$ and $\phi_{m'}(E_{e_2x^2})$ are of type $(i)$ and $(ii)$.\\
Consider the difference
\begin{align*}
    \pm[22-3D_5-7g_5^{\omega_y}].
\end{align*}
In this difference, there exists an element with absolute value of coefficient $3$. Which is not possible.

\medskip
\noindent\textbf{Case 2:} $\phi_{m'}(E_{e_2})$ and $\phi_{m'}(E_{e_2x^2})$ are of type $(i)$ and $(iii)$.
Consider the difference
\begin{align*}
    \pm[2D_5+7g_5^{\omega_y}-17g_5^{\gamma_y}].
\end{align*}
In this difference, there exists an element with absolute value of coefficient $ 2$. Which is not possible.\\
\textbf{case 3:}  $\phi_{m'}(E_{e_2})$ and $\phi_{m'}(E_{e_2x^2})$ are of type $(i)$ and $(iv)$.
Consider the difference
\begin{align*}
    \pm [4D_5+7g_5^{\omega_y}-7(g_5^{\alpha_y}+g_5^{\beta_y})-13g_5^{\gamma_y} ].
\end{align*}
\begin{enumerate}[label=(\roman*)]
    \item $\gamma_y=\omega_y=\alpha_y=\beta_y$.\\
 There exists an element with absolute value of coefficient $16$. Which is not possible.
\item $\gamma_y=\omega_y$, $\gamma_y \neq \alpha_y$ and $\gamma_y=\beta_y$.\\
The absolute value of coefficient $g_5^{\gamma_y}$ is $ 9$. This will contradict Proposition \ref{coff_of_diff_of_vanishing_sum}. Therefore, this is not possible.
\item $\gamma_y=\omega_y$, $\gamma_y \neq \alpha_y$ and $\gamma_y \neq \beta_y$.\\
The absolute value of coefficient $g_5^{\gamma_y}$ is $2$, which is not possible.
\item $\gamma_y \neq \omega_y$.\\
The absolute value of coefficient of $g_5^{\gamma_y}$ is greater than $9$, which is not possible.
\end{enumerate}

\noindent\textbf{Case 4:} $\phi_{m'}(E_{e_2})$ and $\phi_{m'}(E_{e_2x^2})$ are of type $(ii)$ and $(iii)$.\\
Consider the difference
\begin{align*}
    \pm [D_5+17g_5^{\gamma_y}-22].
\end{align*}
If $\gamma_y \equiv 0 \mod{5}$, then the coefficient of $1$ is $\pm 4$. Otherwise, the coefficient of $g_5^{\gamma_y}$ is $\pm 18$. Both are impossible.

\medskip
\noindent \textbf{Case 5:} $\phi_{m'}(E_{e_2})$ and $\phi_{m'}(E_{e_2x^2})$ are of type $(ii)$ and $(iv)$.\\
Consider the difference
\begin{align*}
    \pm[D_5+22-7g_5^{\alpha_x}-7g_5^{\beta_x}-13g_5^{\gamma_x}].
\end{align*}
\begin{enumerate}[label=(\roman*)]
    \item $\gamma_x=\alpha_x=\beta_x$.\\
    If $\gamma_y\equiv 0 \mod{p}$, then the coefficient of $1$ is $\pm 4$, which is not possible.
    Otherwise, the coefficient of $1$ is $23$, which is not possible.
    \item If atleast two of $\alpha_y,~ \beta_y$ and $\gamma_y$ are not equal.
    Then there exists an element with absolute value of coefficient $\geq 6$, which is not possible. 
\end{enumerate}

\medskip
\noindent \textbf{Case 6:} $\phi_{m'}(E_{e_2})$ and $\phi_{m'}(E_{e_2x^2})$ are of type $(iii)$ and $(iv)$.\\
Consider the difference set
\begin{align}
    \pm [2D_5+17g_5^{t_y}-7g_5^{\alpha_y}-7g_5^{\beta_y}-13 g_5^{\gamma_y}].
\end{align}
\begin{enumerate}[label=(\roman*)]
    \item $t_y = \gamma_y=\alpha_y =\beta_y$. \\ 
    In this case, the coefficient of $g_5^{\gamma_y}$ is either $8$ or $-8$. Which is not possible.
    \item $t_y=\gamma_y$, $\gamma_y \neq \alpha_y$, and $\gamma_y =\beta_y$. \\
    In this case, the coefficient of $g_5^{\alpha_y}$ is either   $ 5$ or $-5$. Which is not possible.
    \item $t_y=\gamma_y$, $\gamma_y \neq \alpha_y$, and $\gamma_y \neq \beta_y$.\\
    In this case, coefficient of $g_5^{\gamma_y}$ will be $\pm 6$, which is not possible.
    \item ${t_y\neq \gamma_y}$.\\
    In this case, then there exists an element with absolute value of coefficient $\geq 5$. This is not possible.
\end{enumerate}
Now, we will look at the case when both $\phi_{m'}(E_{e_2})$ and $\phi_{m'}(E_{e_2x^2})$ are of same type and $\phi_{m'}(E_{e_2})-\phi_{m'}(E_{e_2x^2})\neq 0$.\\
If  $\phi_{m'}(E_{e_2})$ and $\phi_{m'}(E_{e_2x^2})$  both are of type $(i)$, then the difference is either $0$ or there exists an element in the difference whose coefficient has absolute value $7$, which is not possible. If $\phi_{m'}(E_{e_2})$ and $\phi_{m'}(E_{e_2x^2})$  both are of type $(ii)$, then there difference will be $0$. Now, consider the case, when both are of type $(iii)$, then the difference is either $0$ or there exists an element with coefficient $17$. 
Now, consider $\phi_{m'}(E_{e_2})$ and $\phi_{m'}(E_{e_2x^2})$ are of type $(iv)$, then the difference is either $0$ or \[\pm[7(g_5^{\alpha_y}+g_5^{\beta_y}-g_5^{\alpha_y'}-g_5^{\beta_y'})+13( g_5^{\gamma_y}- g_5^{\gamma_y'})].\]
Then there exists a coefficient with absolute value greater than equals to $6$. Hence this case will not occur.\\
In the case where the difference is zero, we instead consider
\[
\phi_{m'}(E_{e_2x}) - \phi_{m'}(E_{e_2x^2}),
\]
which is nonzero. By the same analysis, this difference also contains a coefficient whose absolute value exceeds $1$. Hence the result follows.
\end{proof}
\begin{corollary} \label{non exist_for_n=3_m_odd}
    There does not exist a GBF $f :\mathbb{Z}_3^3 \rightarrow \mathbb{Z}_{5\cdot 11^r}$. 
\end{corollary}
\begin{proof}
    Suppose that there exists a GBF from $:\mathbb{Z}_3^3 \rightarrow \mathbb{Z}_{5\cdot 11^r}$. Since the $c$-exponent of $E_x$ divides $5\cdot 11^r$, it follows from Theorem \ref{possible_form_of_e_x_n_3} that $E_x = D_5 g^{\alpha_x} + D_{11}\bigl(g^{\beta_x} + g^{\gamma_x}\bigr)$ for each $x\in G$. However, this contradicts Proposition \ref{resul_excluiding_one_form_of_e_x}. Therefore, there does not exist a $f :\mathbb{Z}_3^3 \rightarrow \mathbb{Z}_{5\cdot 11^r}$. 
\end{proof}

When we apply the homomorphism $\phi_{m'}$ to the other possible forms of $E_x$ arising from a GBF $ f:\mathbb{Z}_3^3 \to \mathbb{Z}_m$, we do not obtain a simplified expression similar to the one obtained in Lemma \ref{special_form_of_E_x on action of phi_m} for $ E_x=D_5g^{\alpha_x}+D_{11}\bigl(g^{\beta_x}+g^{\gamma_x}\bigr)$.
Consequently, it becomes difficult to analyze the remaining possible forms of $E_x$. Using this method, we are only able to exclude $E_x=D_5g^{\alpha_x}+D_{11}\bigl(g^{\beta_x}+g^{\gamma_x}\bigr)$ whenever there exists a GBF $ f:\mathbb{Z}_3^3 \to \mathbb{Z}_m$.

%%%%%%%%%%%%%%%%%%%%%%%%%%%%%%%%%%%%%%%%%%%%%%%%%%%%%%%%%%%%%%%%%%%%%%%%%%%%%%%%%%%%%%%%%%%%%%%%%%%%%%%%%%%%%%%%%%%%%%%%%%%%%%%%%%%%%%%%%%%%%%%%%%%%%%%%%%%%%%%%%%%%%%%%%%%%
\section{Nonexistence results when even $m$ not divisible by $3$} 
This section addresses the remaining case when $m$ is even and not divisible by $3$. We show that no GBF exists when $m=2^k$, and then consider the case $m=2m'$ with $m'$ is odd and not divisible by $3$.
%%%%%%%%%%%%%%%%%%%%%%%%%%%%%%%%%%%%%%%%%%%%%%%%%%%%%%%
\subsection*{Nonexistence result when $m=2^k$}
Suppose that there exists a GBF $f : \mathbb{Z}_3^n \rightarrow \mathbb{Z}_m$, where $m = 2^k$. Then, for each $x \in G \setminus \{1_G\}$, the element $E_x \in \mathbb{N}[C_m]$ is a $v$-sum. Since $2$ is the only prime divisor of $m$, it follows that the $c$-exponent of $E_x$ is equal to $2$. By Proposition~\ref{Discription_of_v-sum}$(ii)$, we have
\[
E_x = D_2 Y,
\]
where $Y \in \mathbb{N}[C_m]$. Consequently, $ 3^n = \lVert E_x \rVert = 2 \lVert Y \rVert$, which is a contradiction since $3^n$ is odd. This gives us the following result.
\begin{theorem} \label{NO_GBF_for_2^k}
There does not exist a GBF  $ f: \mathbb{Z}_3^n \rightarrow \mathbb{Z}_{2^k}$ for any positive integer $n$.
\end{theorem}

\subsection*{Nonexistence results when $m=2m'$ with $m'$ is odd and not divisible by 3}
\begin{lema} \label{expression of E_x}
    Let $f~:~ \mathbb{Z}_3^n \rightarrow \mathbb{Z}_{2m'}$ be a GBF, where $m' = \prod_{j=1}^tp_j^{\alpha_j}$, with primes satisfying $5 \leq p_1 < p_2 < \cdots < p_t$ and $\alpha_j$ are positive integers. Then $t \geq 1$ and $3^n \geq p_1+2$. Moreover, If $n=2$, then $E_x$ is one of the following forms:
    \begin{enumerate}[label=(\roman*)]
        \item $D_2(h_1+h_2)+D_5h_3$,
        \item $D_2(h_1)+D_7h_2$,
    \end{enumerate}
    where $h_i \in C_{2m'}$.
\end{lema}
\begin{proof}
    Assume that $f$ is a GBF. Then $E_x$ is a $v$-sum for every $x \in \mathbb{Z}_3^n\setminus \{0\}$. If $t=0$, then by Theorem~\ref{NO_GBF_for_2^k}, there does not exist a GBF $f : \mathbb{Z}_3^n \rightarrow \mathbb{Z}_2$. Hence, we may assume that $t \geq 1$.
    We can decompose $E_x$ as $\sum_{i=1}^wX_i$, where each $X_i$ is a minimal $v$-sum with reduced exponent $k_i$. Now, we consider the following cases.

    \medskip
    \noindent \textbf{Case 1:} $|\mathcal{P}(k_i)| \geq 4$ for some $1 \leq i \leq w$.\\ From Proposition \ref{Discription_of_v-sum} ($i$), we have 
    \begin{align*}
       \lVert X_i \rVert &\geq 2+(2-2)+(p_1-2)+ (p_2-2)+(p_3-2)\\
        &  \geq 2p_1+p_2 \geq p_1+2.
    \end{align*}
    Thus, $3^n \geq 2p_1+p_2$. For $n=2$, this gives $9 \geq 2p_1+p_2 \geq  10$, which is impossible. Hence, this case does not appear for $n=2$.\\
    Now, assume that $|\mathcal{P}(k_i)| \leq 3$. By Proposition \ref{Discription_of_v-sum} $(vi)$, it follows that $|\mathcal{P}(k_i)|$ is either $1$ or $3$. 

    \medskip
    \noindent \textbf{Case 2:} $|\mathcal{P}(k_i)|=3$ for some $1\leq i \leq w$.\\
    From Proposition \ref{Discription_of_v-sum} $(vi)$, we have
    \begin{align*}
        \lVert X_i \rVert &\geq (2-1)(p_1-1)+(p_2-1)\\
        & \geq (p_1-1) +(p_2-1)~ \geq ~ 2p_1 ~\geq~ p_1+2.
    \end{align*}
    Hence, $3^n \geq 2p_1$. For $n=2$, this implies $9 \geq 2p_1 \geq 10$, which is impossible. Therefore, this case does not occur when $n=2$.

\medskip
\noindent \textbf{Case 3:} $\mathcal{P}(k_i)=1$ for all $1\leq i\leq w$.\\
By Proposition \ref{Discription_of_v-sum} ($v$), we have $X_i = D_{q_i}Y_i$, where $Y_i \in C_{2m'}$. Thus, $E_{x}=\sum_{i=1}^wD_{q_i}Y_i$. If all $q_i$ are equal, then $E_x=D_{q_1}Y$, for some $Y\in \mathbb{N}[C_{2m'}]$, which implies $3^n=q_1\lVert Y \rVert$. Which is impossible since $q_1$ does not divide $3^n$.  Without loss of generality, we may assume that $q_1 \neq q_2$. Thus, $E_x =D_{q_1}Y_1+D_{q_2}Y_2+\sum_{i=3}^wD_{q_i}Y_i$, therefore $3^n=\lVert E_x \rVert \geq q_1+q_2+(w-2)2$.

\medskip
\noindent \textbf{Case 3.1:}  $w > 3$. \\
Then
\begin{align*}
    3^n \geq q_1+q_2+4 \geq p_1+6 \geq p_1+2.
\end{align*}
In this case, $3^n \geq p_1+6$. For $n=2$, this gives $9 \geq 10$, which is a contradiction. Hence, this case does not occur when $n=2$.

\medskip
\noindent \textbf{Case 3.2:} $w=3$. \\
Then
\begin{align*}
    3^n \geq q_1+q_2+2 \geq p_1+4 \geq p_1+2.
\end{align*}
In this case, $E_x=D_{q_1}Y_1+D_{q_2}Y_2+D_{q_3}Y_3$. Therefore, if $n=2$, then $9=q_1+q_2+q_3$. Since $q_1$ and $q_2$ are distinct primes, therefore the only possibility for $(q_1,q_2,q_3) ~\text{is}~(2,5,2)$. It follows that for $n=2$ we have $E_x=D_2(Y_1+Y_3)+D_5Y_2$.

\medskip
\noindent \textbf{Case 3.3:}  $w=2$.\\ 
Then
\[
3^n \geq q_1+q_2 \geq p_1+2.
\] In this case, $E_x = D_{q_1}Y_1+D_{q_2}Y_2$. Therefore, if $n=2$, then  $9=q_1+q_2$. Since $q_1$ and $q_2$ are distinct primes, therefore the only possibility for $(q_1,q_2)$ is $(2,7)$. It follows that for $n=2$, we have $E_x=D_2Y_1+D_7Y_2$. 
This completes the proof.
\end{proof}
\begin{corollary} \label{non_exist_n=1_mis even}
    There does not exist a GBF $f:\mathbb{Z}_3 \to \mathbb{Z}_{2m'}$, where $m'$ is an odd positive integer not divisible by 3.  
\end{corollary}
\begin{proof}
By Lemma \ref{expression of E_x}, if there exists a GBF $f: ~\mathbb{Z}_3 \rightarrow \mathbb{Z}_{2m'}$ then $3 \geq p_1+2$. However, since $p_1 \geq 5$, this inequality cannot hold.  Therefore, there does not exists a GBF from $\mathbb{Z}_3$ to $\mathbb{Z}_{2m'}$.
\end{proof}
Throughout the remainder of this section, we fix $G = \mathbb{Z}_3^2$. Let $f : G \rightarrow \mathbb{Z}_{2m'}$ be a function, and define
\[
G_f := \{ x \in G \mid f(x) \text{ is odd} \}.
\]

If $G_f=\emptyset$ and there exists a GBF $f:\mathbb{Z}_3^n \rightarrow \mathbb{Z}_{2m'}$, exists a GBF $f : \mathbb{Z}_3^2 \rightarrow \mathbb{Z}_{2m'}$, then the function $g : \mathbb{Z}_3^2 \rightarrow \mathbb{Z}_{m'}$ defined by
$ g(x) = \frac{f(x)}{2}$
is also a GBF because $\zeta_{2m'}^{f(x)}$ is $m'^{th}$ root of unity.
\begin{lema} \label{no_gbf_for_empty_G_f}
    Let $G_f = \emptyset$, then there does not exist a GBF $ f:\mathbb{Z}_3^2  \rightarrow \mathbb{Z}_{2m'}$.
\end{lema}
\begin{proof}
Suppose there exists a GBF $f : \mathbb{Z}_3^2 \rightarrow \mathbb{Z}_{2m'}$ with $G_f = \emptyset$. Then $f(x)$ is even for all $x \in G$, and hence the function $g(x) = \frac{f(x)}{2}$ defines a GBF from $\mathbb{Z}_3^2$ to $\mathbb{Z}_{m'}$. However, by Theorem~\ref{Nonexistence_of _gbf_for_odd_m}, no such GBF exists. This is a contradiction. Hence, no GBF from $\mathbb{Z}_3^2$ to $\mathbb{Z}_{2m'}$ exists.
\end{proof}

Define a ring homomorphism $\phi: \mathbb{Z}[G \cdot C_{2m'}] \rightarrow \mathbb{Z}[G]$ by
\begin{align*}
    \phi(x) &= x \quad \text{for all } x \in G, \\
    \phi(g) &= -1.
\end{align*}
Then
\begin{align*}
    \phi(D_f) = \phi \left (\sum_{x\in G}g^{f(x)}x\right )= G-2G_f.
\end{align*}
Consequently,
\begin{align*}
    \phi(D_f D_f^{(-1)}) 
    &= \phi(D_f)\,\phi(D_f^{(-1)}) \\
    &= (G - 2G_f)(G - 2G_f^{(-1)}) \\
    &= G^2 - 2G G_f^{(-1)} - 2G G_f + 4 G_f G_f^{(-1)} \\
    &= (|G| - 4|G_f|)\,G + 4 G_f G_f^{(-1)}.
\end{align*}
Write $G_fG_f^{(-1)}=|G_f|+\sum \limits_{x(\neq 1_G)\in G}a_x x$, where $a_x$ are non-negative integers. Substituting this into the above expression, we obtain
\begin{align}\label{1st_expresion_of_phi}
    \phi(D_fD_f^{-1})=|G| + \sum_{x(\neq 1)\in G}(|G|-4|G_f|+4a_x)x
\end{align}
\noindent On the other hand, from Equation \eqref{D_fD_f-1 =_sum_of E_x}, we have
\begin{align}\label{2nd_expression_ofPhi}
    \phi(D_fD_f^{(-1)})=|G|+\sum_{x(\neq 1)\in G} \phi(E_x)x
\end{align}
Comparing Equations \eqref{1st_expresion_of_phi} and \eqref{2nd_expression_ofPhi}, we have $\phi(E_x)=|G|-4|G_f|+4a_x$.
Finally, by replacing $f$ with $f + m'$ if necessary, we may assume that
\[
|G_f| \leq \frac{|G| - 1}{2}.
\]
\begin{theorem}\label{non_exist_n=2and_m_is_even}
    There does not exist a GBF $ f:\mathbb{Z}_3^2 \to \mathbb{Z}_{2m'}$, where $m' = \prod_{j=1}^tp_j^{\alpha_j}$ with primes satisfying $5 \leq p_1 < p_2 < \cdots < p_t$ and $\alpha_j$ are positive integers.
    
\end{theorem}
\begin{proof}
   Assume that there exists a GBF $f : \mathbb{Z}_3^2 \rightarrow \mathbb{Z}_{2m'}$. Then $E_x$ is a $v$-sum for every $x\in G\setminus\{1_G\}$. 
   By Lemma~\ref{expression of E_x}, the possible forms of $E_x$ are
   \begin{enumerate}[label=(\roman*)]
       \item $E_x = D_2(g^{\alpha_x}+g^{\beta_x})+D_5g^{\gamma_x}, \quad \text{and} \hspace{2mm} \phi(E_x)=\pm 5$,
       \item $E_x = D_2g^{\alpha_x}+D_7g^{\beta_x}, \quad \text{and} \hspace{2mm}\phi(E_x)=\pm 7$.
   \end{enumerate}
Hence, the possible values of $\phi(E_x)$ are $\pm 5,\pm 7$. 
   From Lemma \ref{no_gbf_for_empty_G_f}, we may assume that $|G_f| \geq 1$. Adding $m'$ to $f$ if necessary, we may assume that $|G_f| \leq \frac{|G|-1}{2}=4$. We may also assume that $f(1_G)$ is odd. We now consider the following cases.

   \medskip
   \noindent \textbf{Case 1 :} $|G_f|=4$.

   \medskip
   \noindent \textbf{Case 1.1} $G_f = \{1, u, v, w\}$, where $u,v,w$ are distinct elements such that none of $u = v^2$, $v = w^2$, or $u = w^2$ holds. Considering $G_f$ as an element of the group ring $\mathbb{Z}[G\cdot C_m]$ we have
    Then
   \begin{align*}
       G_fG_f^{(-1)} &= (1+u+v+w)(1+u^2+v^2+w^2)\\
       &=4+ u+v+w +u^2+v^2+w^2+uv^2+uw^2+vu^2+vw^2+wu^2+wv^2.
   \end{align*}
Since this expression contains more than $|G| = 9$ terms, some terms must coincide. The term involving $2$ variables cannot be equal to another term with $2 $ variables since this would contradict the condition of this case. Without loss of generality, we may assume that $u$ coincides with one of the terms on the right-hand side. Since $u \neq v^{2},v \neq w^{2}, \text{and} \hspace{2mm} u \neq w^2 $, it follows that either $u=vw^2$ or $u=wv^2$. In both the cases $a_u=2$.  Therefore,
\[
\phi(E_u) = |G| - 4|G_f| + 4a_u = 9 - 16 + 8 = 1,
\]
which is impossible, since $\phi(E_x) \in \{\pm 5, \pm 7\}$.

\medskip
\noindent \textbf{Case 1.2} : $G_f =\{1,v,v^2,w\}$. \\
Then
\begin{align*}
    G_fG_f^{-1} &= (1+v+v^2+w)(1+v+v^2+w^2)\\
    &=  4+3v+3v^2+w+w^2+vw^2+vw+v^2w^2+v^2w.
\end{align*}
As $a_w=1$, which implies $\phi(E_w)=-3$. This is impossible, since the only admissible values of $\phi(E_x)$ are $\pm 5$ and $\pm 7$.

\medskip
\noindent \textbf{Case 2:} $|G_f|=3$.\\
\textbf{Case 2.1}: $G_f=\{1,v,w\}$, where  $v$ and $w$ are distinct element of $G$ with $w \neq v^2$.
\begin{align*}
    G_fG_f^{(-1)}&=(1+v+w)(1+v^2+w^2)\\
    &=3+v+w+v^2+w^2+vw^2+wv^2.
\end{align*} 
Clearly, $a_v=1$, and hence $\phi(E_v)=|G|-4|G_f|+4a_v=1$. Again, this is impossible, since the only possible values of $\phi(E_x)$ are $\pm 5$ and $\pm 7$.

\medskip
\noindent \textbf{Case 2.2}: $G_f=\{1,v,v^2\}$.
\begin{align*}
    G_fG_f^{(-1)}&=(1+v+v^2)(1+v+v^2)\\
    &=3+3v+3v^2.
\end{align*}
Hence $a_v = 3$, and
\[
\phi(E_v) = 9 - 12 + 12 = 9,
\]
which is again impossible.

\medskip
\noindent \textbf{Case 3:} $|G_f|=2$.\\
Let $G_f=\{1,v\}$, then $G_fG_f^{(-1)}=2+v+v^2$. We have $a_x=0 $ for all $x \in G\setminus{\{1,v,v^2\}}$. Therefore, $\phi(E_x)=|G|-4|G_f|+4a_x=1$. This is again impossible, since the possible values of $\phi(E_x)$ are $\pm 5$ and $\pm 7$.

\medskip
\noindent\textbf{Case 4:} $|G_f|=1$.\\
In this case, $G_fG_f^{(-1)}=1_G$, and hence $a_x=0$ for all $x\in G\setminus \{1_G\}$. This implies that $\phi(E_x)=5$ for all $x\in G\setminus \{1_G\}$. From Lemma \ref{expression of E_x}, it follows that $E_x=D_2(g^{\alpha_x}+g^{\beta_x})+D_5g^{\gamma_x}$ for all $x\in G\setminus \{1_G\}$, where $g^{\alpha_x},g^{\beta_x}, \hspace{2mm}\text{and} \hspace{2mm}g^{\gamma_x}$ are elements of $ C_{2m'}$, and $\gamma_x$ is an even integer. \\
Therefore, if a GBF $f : \mathbb{Z}_3^2 \to \mathbb{Z}_{2m'}$ exists, then necessarily $|G_f| = 1$ and
\[
E_x = D_2\bigl(g^{\alpha_x} + g^{\beta_x}\bigr) + D_5 g^{\gamma_x}
\quad \text{for all } x \in G \setminus \{1_G\},
\]
with $\gamma_x$ even.

Note that the possible value of $c$ - exponent of $E_x$ is $10$ for every $x\in G\setminus \{1_G\}$. By Lemma \ref{existence_of_the_gbf_to_lower_order}, we may assume that there exists a GBF $f :\mathbb{Z}_3^2\rightarrow \mathbb{Z}_{2 \cdot 5^a}$. Let $g_2$ and $g_{5^a}$ be elements of $C_{2\cdot5^a}$ of order $2$ and $5^a$, respectively.
Then $g=g_2g_{5^a}$ is a generator of $C_{2\cdot 5^{a}}$. Define a ring homomorphism $\psi:\mathbb{Z}[G\cdot C_{2\cdot 5^a}] \rightarrow \mathbb{Z}[C_{5^a}]$ such that $\psi(g_2)=-1$, $\psi(g_{5^a})= g_{5^a}$, and $\psi(x)=1$ for all $x\in G$. Then
\begin{align*}
    \psi(D_f)&=\sum_{x\in G}\psi(g_2g_5^a)^{f(x)} \psi(x) = \sum_{x\in G}(-g_{5^a})^{f(x)}\\
    & = \sum_{\substack{x\in G\\x\neq 1_G} }(g_{5^a})^{f(x)}-g_{5^a}^{f(1_G)},
\end{align*}
since $f(x)$ is even for $x \neq 1_G$ and odd for $x = 1_G$.
Consequently,
\begin{align*}
  \psi(D_f)\psi(D_f^{(-1)}) & = \left(\sum_{\substack{x\in G\\x\neq 1_G} }(g_{5^a})^{f(x)}-g_{5^a}^{f(1_G)}\right)\left(\sum_{\substack{l\in G\\l\neq 1_G} }(g_{5^a})^{-f(l)}-g_{5^a}^{-f(1_G)}\right)\\
  &= 9+ \sum_{\substack{x\neq l\\x,l \neq 1_G}}g_{5^{a}}^{f(x)-f(l)} -\sum_{\substack{x\in G\\x \neq 1_G}}g_{5^a}^{f(x)-f(1_G)} -\sum_{\substack{l\in G\\l \neq 1_G}}g_{5^a}^{f(1_G)-f(l)}.
\end{align*}
Note that $f(x)$ is even for all $x \in G \setminus \{1_G\}$. Therefore, for $x,l \in G \setminus \{1_G\}$, the difference $f(x) - f(l)$ is even, whereas $f(x) - f(1_G)$ is odd for all $x \in G \setminus \{1_G\}$. It follows that some coefficients in $\psi(D_f)\psi(D_f^{(-1)})$ are negative. On the other hand, for all $x \in G \setminus \{1_G\}$, we have $E_x = D_2(g^{\alpha_x}+g^{\beta_x})+D_5g^{\gamma_x}$ where $\gamma_x$ is even. Since $D_5=1+g_{5^a}+g_{5^a}^2+g_{5^a}^3+g_{5^a}^4$, it follows that $\psi(E_x)=D_5~g_{5^a}^{\gamma_x}$. Using Equation \eqref{D_fD_f-1 =_sum_of E_x}, we have 
\begin{align}\label{psi(D_f)\psi(D_f)_in _terms_of_psi(E_x)}
    \psi(D_f)\psi(D_f^{(-1)})&=9+\sum_{\substack{x\in G\\x\neq 1_G}}\psi(E_x) \nonumber\\
    &=9+ \sum_{\substack{x\in G\\x\neq 1_G}}D_5 ~g_{5^a}^{\gamma_x}. \nonumber
\end{align}
    This expression shows that all coefficients of $\psi(D_f)\psi(D_f^{(-1)})$ are non-negative integers. This contradicts the earlier observation that some coefficients must be negative.  
 Therefore, $|G_f| \neq 1$. Hence, there does not exist GBF from $\mathbb{Z}_3^n$ to $\mathbb{Z}_{2m'}$.
\end{proof}

\section{Conclusion}
In this concluding section, we summarize our main results on generalized bent functions from $\mathbb{Z}_3^n$ to $\mathbb{Z}_m$. We observed that GBFs exist whenever $m$ is divisible by $3$. On the other hand, when $m$ is not divisible by $3$, we obtained several nonexistence results. In particular, for $n=1,2$, we proved that there does not exist any GBF when $m$ is odd and not divisible by $3$. We also showed that there does not exist a GBF
$
f:\mathbb{Z}_3^3 \rightarrow \mathbb{Z}_{5\cdot 11^r}
$.

We further investigated the case where $m$ is even and not divisible by $3$. In particular, we proved that there does not exist a GBF
$f:\mathbb{Z}_3^n \rightarrow \mathbb{Z}_{2^k}$ for any positive integer $n$. Moreover, we established the nonexistence of GBFs from $\mathbb{Z}_3$ to $\mathbb{Z}_{2m'}$ and from $\mathbb{Z}_3^2$ to $\mathbb{Z}_{2m'}$, where $m'$ is odd and not divisible by $3$.

However, the cases where $m$ is odd and not divisible by $3$, and where $m$ is even and not divisible by $3$, are still not completely understood for larger values of $n$. It remains open whether GBFs exist in these cases or whether stronger nonexistence results can be obtained.

Future work may focus on extending the methods developed in this paper to the remaining unresolved cases. It would also be interesting to study explicit constructions of GBFs whenever they exist and to investigate similar problems for other finite abelian groups. Overall, this work provides a better understanding of GBFs on $\mathbb{Z}_3^n$ and suggests several directions for future research.

\section*{Acknowledgements}

\noindent The first author is supported by a Senior Research Fellowship from CSIR, Government of India (File No. 09/1020(15619)/2022-EMR-I).

\noindent\textbf{Data Availability :}  No data was gathered or used in this paper, so a ``data availability statement'' is not applicable.

\noindent\textbf{Conflict of interest :} The author states that there is no conflict of interest.

\end{document}